\documentclass[a4paper,12pt]{article}
\usepackage{amsmath}
\usepackage{amsfonts}
\usepackage{amssymb}
\usepackage{amsthm}
\usepackage{graphicx}
\usepackage{ifthen}
\usepackage{verbatim}
\usepackage{amsmath}
\usepackage{setspace}
\usepackage{amsfonts}
\usepackage{amssymb}
\usepackage{amsthm}
\usepackage{lscape}
\usepackage{setspace}
\usepackage{listings}
\usepackage{amsmath}
\usepackage{amsfonts}
\usepackage{amssymb}
\usepackage{amsthm}
\usepackage{ifthen}
\usepackage{verbatim}
\usepackage{bbm}
\newtheorem{thm}{Theorem}[section]

\newtheorem{coro}[thm]{Corollary}

\newtheorem{defn}[thm]{Definition}
\newtheorem{rem}[thm]{Remark}
\newtheorem{prop}[thm]{Proposition}

\newcommand{\levy}{{L\'evy }}
\newcommand{\cadlag}{{c\`adl\`ag }}
\def\1{\mathbbm{1}}
\def\R{\mathbb{R}}
\def\P{\mathbb{P}}

\def\Q{\mathbb{P}}
\def\E{\mathbb{E}}
\def\N{\mathbb{N}}

\def\d{\,\mathrm{d}}
\def\dd{\mathrm{d}}

\def\var{\mathrm{Var}}
\def\cov{\mathrm{Cov}}

\def\b{\beta}

\def\s{\sigma}
\def\law{\overset{\textnormal{law}}{=}}
\def\tp{\top}
\def\F{\mathcal{F}}

\def\Borel{\mathcal{B}(\R)}

\def\t{\tau}
\def\D{\Delta}

\def\ab{\boldsymbol{\alpha}}

\def\xib{\boldsymbol{\xi}}
\def\sib{\boldsymbol{S}}

\def\sigmab{\boldsymbol{\sigma}}
\def\xb{\mathbf{x}}
\def\yb{\mathbf{y}}
\def\mb{\mathbf{m}}
\def\sumel{\mathbf{1}\cdot}

\begin{document}

\title{Generalised Liouville Processes and their Properties}

\author{Edward Hoyle\thanks{AHL Partners LLP, Man Group plc,
		London, EC4R 3AD, United Kingdom}, 
	Levent Ali Meng\"ut\"urk\thanks{Department of Mathematics, University College London, WC1E 6BT, United Kingdom}}

\maketitle

\begin{abstract}
We define a new family of multivariate stochastic processes over a finite time horizon that we call Generalised Liouville Processes (GLPs). 
GLPs are Markov processes constructed by splitting \levy random bridges into non-overlapping subprocesses via time changes.
We show that the terminal values and the increments of GLPs have generalised multivariate Liouville distributions, justifying their name. 
We provide various other properties of GLPs and some examples.
\end{abstract}
{\bf Keywords:} \levy process; \levy bridge; Markov process; multivariate Liouville distribution\\
{\bf MSC2010:} {60}

\section{Introduction} 
\levy random bridges (LRBs) -- \levy processes conditioned to have a fixed marginal law at a fixed future date -- have been applied to various problems in credit risk modelling, asset pricing and insurance (see, for example, \cite{11,9,12,10,13}).
In \cite{15}, the authors present a bivariate insurance reserving model by splitting an LRB (in this case based on the 1/2-stable subordinator) in two.
The two subprocesses are transformed to span the same time horizon, and are used to model the accumulation of insurance claims.
In a similar fashion, the present authors constructed in \cite{14} two classes of multivariate process by splitting and transforming an LRB based on the gamma process.
The first class, Archimedean survival processes, provide a natural link between stochastic processes and Archimedean copulas, and was applied to a copula interpolation problem.
The second, more general, class was the class of Liouville processes, so named because the finite dimensional distributions of a Liouville process are multivariate Liouville distributions \cite{2,1,3,4}.
This more general class was applied to the joint modelling of realised variance for two stock indices.

We extend the splitting and transformation mechanism to a general LRB to create what we call a \emph{generalised Liouville process} (GLP).
We show that the sum of coordinates of GLPs are one-dimensional LRBs, and prove that the finite dimensional distributions of GLPs are generalised multivariate Liouville distributions as defined in \cite{6}.
We show that GLPs are Markov processes and that there exists a measure change under which the law of an $n$-dimensional GLP is that of a vector of $n$ independent \levy processes.
We prove that any integrable GLP admits a canonical semimartingale representation with respect to its natural filtration. We also show that GLPs are multivariate harnesses. 
We prove that GLPs satisfy the weak Markov consistency condition, but not necessarily the strong Markov consistency condition. 
Similarly, we introduce what we call weak and strong semimartingale consistency properties, and show that GLPs have the former, but not necessarily the latter.
The class of GLPs contains as special cases: Archimedean survival processes, Liouville processes, and the bivariate process based on the 1/2-stable subordinator. 

Throughout much of this work, we focus on processes taking continuous values.
However, although details are omitted, many results are straightforward to extend to processes on a lattice.
Indeed, later we provide examples of both a continuous and a discrete GLP. 
More specifically, we consider what we call Brownian Liouville processes and Poisson Liouville processes, and present some of their special characteristics.


\section{Preliminaries}
Throughout this work, for a vector $\mathbf{x}\in\R^n$, we denote the sum of its coordinates by $\sumel \mathbf{x}=\sum_i x_i$. 
We work on a probability space $(\Omega,\mathcal{F},\mathbb{P})$ equipped with a filtration $\{\mathcal{F}_{t}\}_{t\geq0}$. 
We fix a finite time horizon $t\in[0,T]$ for some $T<\infty$ and assume $\{\mathcal{F}_{t}\}_{0\leq t \leq T}$ and all its sub-filtrations are right-continuous and complete. 
Unless stated otherwise, every stochastic process is \cadlag with a state-space that is a continuous subspace of $(\mathbb{R}^n,\mathcal{B}(\mathbb{R}^n))$ for some $n\in\mathbb{N}_+$, where $\mathcal{B}(\mathbb{R}^n)$ is the Borel $\sigma$-field.

Let $\{X_{t}\}_{t\geq 0}$ be a \levy process taking values in $\mathbb{R}$, such that the law of $X_{t}$ is absolutely continuous with respect to the Lebesgue measure for every $t\in[0,T]$.
In this case the density $f_t$ of $X_t$ exists and satisfies the Chapman-Kolmogorov convolution identity $f_{t}(x)=\int_{\mathbb{R}}f_{t-s}(x-y)f_{s}(y)\d y$, for $0<s<t\leq T$ and $x\in\R$.
Having independent and stationary increments, the finite-dimensional law of $\{X_{t}\}_{0\leq t \leq T}$ is given by
\[
\mathbb{P}(X_{t_{1}}\in \d x_{1},\ldots, X_{t_{m}}\in \d x_{m})=\prod^{m}_{i=1}f_{t_{i}-t_{i-1}}(x_{i}-x_{i-1})\d x_{i},
\]
for $m\in\mathbb{N}_{+}$, $0<t_{1}<\ldots <t_{m}<T$ and $x_{1},\ldots,x_n\in\mathbb{R}$.

A \levy bridge is a \levy process conditioned to take some fixed value at a fixed future time. 
Since \levy processes are homogenous strong Markov processes, the definition of their bridges can be formalised in terms of Doob $h$-transformations.
See \cite{5} for further details on the bridges of Markov processes.
Let $\{X^{(z)}_{t,T}\}_{0\leq t \leq T}$ be a bridge of $\{X_{t}\}_{0\leq t \leq T}$ to the value $z\in\mathbb{R}$ at time $T$, where $0<f_{T}(z)<\infty$. 
The transition density of $\{X^{(z)}_{t,T}\}_{0\leq t < T}$ is given by the following Doob $h$-transform of the transition density of $\{X_{t}\}_{0\leq t < T}$:
\begin{align}
  \mathbb{P}(X^{(z)}_{t,T}\in \d x | X^{(z)}_{s,T}=y)=\frac{h_t(x)}{h_s(y)}f_{t-s}(x-y)\d x,
  \label{eq:ratio}
\end{align}
for $0\leq s <t <T$, where $h_t(x)=f_{T-t}(z-x)$.
Note that $\{h_t\}_{0\leq t < T}$ defined as such is harmonic with respect to $\{X_{t}\}_{0\leq t < T}$.
Note also that $\Q(0<h_t(X^{(z)}_{t,T})<\infty)=1$ for all $0\leq t < T$, so the ratios of densities in (\ref{eq:ratio}) are almost surely well defined (this is discussed in the remark following Proposition 1 of \cite{5}).
Similar ratios feature throughout this work and are likewise almost surely well defined, and we may pass by them without further comment.

\emph{\levy random bridges} (LRBs) are an extension of \levy bridges.
Their interpretation in \cite{13} is as a bridge to an arbitrary random variable at time $T$, rather than a fixed value.
A process $\{L_{t}\}_{0\leq t \leq T}$ is an LRB with generating law $\nu$ if it satisfies: (i) $L_{T}$ has marginal law $\nu$, (ii) there exists a \levy process $\{X_{t}\}_{0\leq t \leq T}$ such that the density $f_t$ of $X_t$ exists for all $t\in(0,T]$, (iii) $\nu$ concentrates mass where $f_{T}$ is positive and finite $\nu$-a.s., (iv) for all $m\in\mathbb{N}_{+}$, every $0<t_{1}<\ldots<t_{m}<T$, every $(x_{1},\ldots,x_{m})\in\mathbb{R}^{m}$, and $\nu$-a.e. $z$,
\[
	\mathbb{P}(L_{t_{1}}\leq x_{1},\ldots, L_{t_{m}}\leq x_{m} | L_{T}=z)=\mathbb{P}(X_{t_{1}}\leq x_{1},\ldots, X_{t_{m}}\leq x_{m} | X_{T}=z).
\]
The finite-dimensional distribution of $\{L_{t}\}_{0\leq t \leq T}$ is given by
\begin{align} \label{eq:LRBfinitedimensionaldistributionglp}
\mathbb{P}(L_{t_{1}}\in \d x_{1},\ldots, L_{t_{m}}\in \d x_{m}, L_{T}\in \d z)=\prod^{m}_{i=1}(f_{t_{i}-t_{i-1}}(x_{i}-x_{i-1})\d x_{i})\vartheta_{t_{m}}(\dd z;x_{m}),
\end{align}
where $\vartheta_{0}(\dd z;y)=\nu(\dd z)$ and $\vartheta_{t}(\dd z;y)=\nu(\dd z)f_{T-t}(z-y)/f_{T}(z)$ for $t\in(0,T)$. 
It follows that LRBs are Markov processes with stationary increments, where the transition law of $\{L_{t}\}_{0\leq t \leq T}$ is
\begin{align} \label{eq:levytransitionprob}
&\mathbb{P}(L_{T}\in \d z | L_{s}=y)=\frac{\vartheta_{s}(\dd z;y)}{\vartheta_{s}(\mathbb{R};y)}, \notag \\
&\mathbb{P}(L_{t}\in \d x | L_{s}=y)=\frac{\vartheta_{t}(\mathbb{R};x)}{\vartheta_{s}(\mathbb{R};y)}f_{t-s}(x-y)\d x,
\end{align}
for $0\leq s < t$.
We note that the finite-dimensional distributions of LRBs with discrete state-spaces have similar transition probabilities given in terms of probability mass functions (for details see \cite{13}).
The extension of many later results to discrete processes follows from this. 
\begin{rem}
\label{harmonicinit}
Note that (\ref{eq:levytransitionprob}) is also a Doob $h$-transform of the transition density of $\{X_{t}\}_{0\leq t < T}$, and $\{\vartheta_{t}(\mathbb{R};X_t)\}_{0\leq t < T}$ is a positive $(\F_t^{X},\P)$-martingale, where $\F_t^{X}=\sigma(\{X_u\}: 0\leq u \leq t)$.
\end{rem}

Let $X_1,\ldots,X_n$ be random variables taking values in $\mathbb{R}$ with a joint density of the form
\begin{equation}
	p\left( \sum_{i=1}^n x_i \right)\prod_{i=1}^n \phi_{a_i}(x_i),
	\label{eq:liouville_density}
\end{equation}
where $a_1,\ldots,a_n>0$ are parameters, and the set of functions $\{\phi_a:a>0\}$ satisfies the convolution property $\phi_{a}*\phi_{b}=\phi_{a+b}$.
In \cite{6}, this is referred to as a ``Liouville density function''.
Indeed, according to the definition given in \cite{6}, $(X_1,\ldots,X_n)$ then has a \emph{Liouville distribution}, although we prefer to refer to this as the \emph{generalised Liouville distribution} to distinguish it from the original and special case that $\{\phi_a\}$ are gamma densities (see \cite{2,1,3,4}).
The actual definition of the generalised Liouville distribution given in \cite{6} replaces the functions $\{\phi_{a}\}$ with measures, and so it includes examples where the joint density may not exist.
For our purposes, it is convenient to relax (\ref{eq:liouville_density}) in a different way.
We keep $\{\phi_a\}$, but replace the function $p$ with a measure $\nu$:
\begin{defn}
	Let $X_1,\ldots,X_n$ be random variables taking values in $\mathbb{R}$, $\nu:\mathcal{B}(\mathbb{R})\rightarrow\mathbb{R}_+$ be a probability law, and $\mathcal{A}=\{\phi_a: 0<a\leq A < \infty)\}$ be a family of functions satisfying the convolution property: $\phi_{a}*\phi_{b}=\phi_{a+b}$, for $a+b\leq A$. Then $(X_1\ldots,X_n)$ has a \emph{generalised multivariate Liouville distribution} if its joint probability law is of the form
	\begin{align}
		\P(X_1\in \dd x_1, \ldots, X_{n-1}\in \dd x_{n-1}, &\sum_{i=1}^n X_{i}\in \dd z) \notag \\
		&= 
		\frac{\phi_{a_n}(z-\sum_{i=1}^{n-1}x_i)\nu\left(\dd z \right)}{\phi_{\sumel \mathbf{a}}(z)} \prod_{i=2}^n \phi_{a_i}(x_i) \d x_i,
		\label{eq:liouville_distribution}
	\end{align}
	for $x_1,\ldots,x_n\in \mathbb{R}$, $\phi_{a_1},\ldots,\phi_{a_n}\in\mathcal{A}$, $\mathbf{a}=(a_1,\ldots,a_n)^\tp\in\R^n_+$, $\sumel \mathbf{a}\leq A$.
\end{defn}
\begin{rem}
	Writing $B+x=\{y:y-x\in B\}$, for $B\subset\mathbb{R}$ and $x\in\mathbb{R}$, then (\ref{eq:liouville_distribution}) is equivalent to
	\begin{align}
		&\P\left(X_1\in \dd x_1, \ldots, X_{n-1}\in \dd x_{n-1}, X_{n}\in B\right) \nonumber \\
			&\qquad =\prod_{i=2}^n \left(\phi_{a_i}(x_i) \d x_i\right)\int_{z\in B+\sum_{i=1}^{n-1}x_i} \frac{\phi_{a_n}(z-\sum_{i=1}^{n-1}x_i)}{\phi_{\sumel \mathbf{a}}(z)} \nu\left(\dd z \right) \nonumber \\
			&\qquad =\prod_{i=2}^n \left(\phi_{a_i}(x_i) \d x_i\right)\int_{x_n\in B} \frac{\phi_{a_n}(x_n)}{\phi_{\sumel\mathbf{a}}(\sum_ix_i)}\nu\left(\sum_{i=1}^{n-1}x_i + \dd x_n \right). \label{eq:liouville_remark}
	\end{align}
	Furthermore, if $\nu$ admits a density $p$, then (\ref{eq:liouville_remark}) can be written in the form of a Liouville density:
	\begin{equation*}
		\P\left(X_1\in \dd x_1, \ldots, X_{n-1}\in \dd x_{n-1}, X_{n}\in \dd x_n\right)= 
			\frac{p\left(\sum_{i}x_i\right)}{\phi_{\sumel \mathbf{a}}(\sum_ix_i)}\prod_{i=1}^n \left(\phi_{a_i}(x_i) \d x_i\right).
	\end{equation*}
\end{rem}

\section{Generalised Liouville processes}
To construct a GLP, we start with a ``master'' LRB $\{L_t\}_{0\leq t \leq u_n}$ for $u_n\in\R_+$ and $n\geq2$, where $L_{u_n}$ has marginal law $\nu$. 
We assume that $\nu$ has no continuous singular part and split $\{L_t\}_{0\leq t \leq u_n}$ into $n$ non-overlapping subprocesses.

\begin{defn} \label{def:GLP}
For $m_1,\ldots, m_n >0$ ($n\geq 2$), define the strictly increasing sequence $\{u_{i}\}^{n}_{i=1}$ by $u_{0}=0$ and $u_{i}=u_{i-1}+m_{i}$ for $i=1,\ldots,n$.
Then a process $\{\xib_{t}\}_{0\leq t \leq 1}$ is an \emph{$n$-dimensional generalised Liouville Process (GLP)} if
\begin{align*}
\{\xib_t\}_{0\leq t \leq 1} \law \left\{\left(L_{tm_1}-L_{0},\ldots,L_{tm_i+u_{i-1}}-L_{u_{i-1}},\ldots,L_{tm_n+u_{n-1}}-L_{u_{n-1}}\right)^\tp\right\}_{0\leq t \leq 1},
\end{align*}
for some LRB $\{L_t\}_{0\leq t \leq u_n}$ with generating law $\nu$.
We say that the \emph{generating law} of $\{\xib_{t}\}_{0\leq t \leq 1}$ is $\nu$ and the \emph{activity parameter} of $\{\xib_{t}\}_{0\leq t \leq 1}$ is $\mb=(m_1,\ldots,m_n)^\tp$.
\end{defn}
We have restricted the definition of GLPs to the time horizon $[0,1]$ for convenience.
It is straightforward to generalise to an arbitrary closed time horizon.
Each coordinate $\{\xi_t^{(i)}\}_{0\leq t \leq 1}$ of $\{\xib_t\}_{0\leq t \leq 1}$ is a subprocess of an LRB. 
Since subprocesses of LRBs are themselves LRBs (see \cite{13}), GLPs form a multivariate generalisation of LRBs. 
For the rest of the paper, we let $\{\xib_t\}_{0\leq t \leq 1}$ be an $n$-dimensional GLP with generating law $\nu$, and $\{L_t\}_{0\leq t \leq u_n}$ be the master process of $\{\xib_t\}_{0\leq t \leq 1}$. 
In addition, we denote the filtration generated by $\{\xib_t\}_{0\leq t \leq 1}$ by $\{\F_t^{\xib}\}_{0\leq t \leq 1} \subset \{\F_t\}_{0\leq t \leq 1}$. Explicitly, we have $\F_t^{\xib}=\sigma(\{\xib_u\}: 0\leq u \leq t)$. 
\begin{rem}
	\label{remasp2}
	The bivariate model of insurance claims based on the 1/2-stable subordinator proposed in \cite{15} is a GLP.
\end{rem}
\begin{rem}
	\label{remasp}
	Liouville processes and Archimedean survival processes, as introduced in \cite{14}, form a subclass of GLPs.
	In Definition \ref{def:GLP}, if the LRB $\{L_t\}_{0\leq t \leq u_n}$ is a gamma random bridge with unit activity parameter, then we have a Liouville process. If we further fix $m_i = 1$ for $i=1,\ldots, n$, then we have an Archimedean survival processes.
\end{rem}

\begin{prop} \label{eq:finitedimensionalGLPincrementcharacterise}
The following hold for any GLP $\{\xib_t\}_{0\leq t \leq 1}$:
\begin{enumerate}
\item The increments of $\{\xib_t\}_{0\leq t \leq 1}$ have a generalised multivariate Liouville distribution.
\item The terminal value $\xib_1$ has a generalised multivariate Liouville distribution.
\end{enumerate}
\end{prop}
\begin{proof}
See Appendix.
\end{proof}
In what follows, we define a family of unnormalised measures $\{\theta_t\}_{0\leq t < 1}$, such that
\begin{align}
	\label{unormmeasures}
	\theta_0(B;x)&=\nu(B), \notag
	\\	\theta_t(B;x)&=\int_B \frac{f_{\sumel \mathbf{m}(1-t)}(z-x)}{f_{\sumel \mathbf{m}}(z)} \, \nu(\dd z),
\end{align}
for $t\in[0,1)$, $x\in\R$ and $B\in\Borel$.
We also write $\Theta_t(x)=\theta_t(\mathbb{R};x)$.
We define $R_t$ to be the sum of coordinates of $\xib_t$:
\begin{equation*}
	R_t=\sum_{i=1}^n \xi^{(i)}_t = \sumel \xib_t.
\end{equation*}
\begin{prop}
\label{markovtransitionprob}
The GLP $\{\xib_t\}_{0\leq t \leq 1}$ is a Markov process with the transition law given by
\begin{multline*}
\Q\left(\left. \xi_1^{(1)}\in\dd z_1,\ldots, \xi_1^{(n-1)}\in\dd z_{n-1},\xi_1^{(n)}\in B  \,\right| \xib_s=\mathbf{x} \right)=
		\\ \frac{\theta_{\t(s)}(B+\sum_{i=1}^{n-1}z_i;x_n+\sum_{i=1}^{n-1}z_i)}{\Theta_s(\sumel \xb)}
									\prod_{i=1}^{n-1}f_{(1-s)m_i}(z_i-x_i)\d z_i,
\end{multline*}
and
\begin{equation}
\label{eq:GLP2}
\Q\left( \xib_t\in \d\mathbf{y}  \,|\, \xib_s=\mathbf{x} \right)=
					\frac{\Theta_{t}(\sumel \yb)}{\Theta_s(\sumel \xb)}
					\prod_{i=1}^{n}f_{(t-s)m_i}(y_i-x_i) \d y_i, 
\end{equation}
where $\mathbf{x}, \mathbf{y}\in\R^n$, $\t(t)=1-m_n(1-t)/(\sumel \mathbf{m})$, $0\leq s<t<1$, and $B\in\Borel$.
\end{prop}
\begin{proof}
See Appendix.
\end{proof}

\begin{rem}
	From Proposition \ref{markovtransitionprob}, if the generating law $\nu$ admits a density $p$, we get a neater transition law to the terminal value, given by
	\begin{equation*}
	\label{nuwithdensity}
	\Q\left( \xib_1 \in \d\mathbf{z} \,|\, \xib_s=\mathbf{x} \right)=
	\frac{p(\sumel \mathbf{z})}{\Theta_s(\sumel \mathbf{x}) f_{\sumel \mathbf{m}}(\sumel \mathbf{z})}
	\prod_{i=1}^n f_{(1-s)m_i}(z_i-x_i) \d z_i.
	\end{equation*}
\end{rem}

\begin{rem}
\label{harmonic}
Our definition of GLPs is somewhat heuristic.
A formal definition is possible through a Doob $h$-transform, since $\{\Theta_{t}\}_{0\leq t < 1}$ is harmonic to a \levy process $\{\mathbf{X}_t\}_{t\geq 0}$ taking values in $\R^n$ with marginal density $g_t(\xb) = \prod_{i=1}^{n}f_{m_i t}(x_i)$.
To see this, note that we can alternatively write (\ref{eq:GLP2}) as
\[
\Q\left( \xib_t\in \d\mathbf{y}  \,|\, \xib_s=\mathbf{x} \right)=
					\frac{\tilde{\Theta}_{t}(\yb)}{\tilde{\Theta}_s(\xb)}
					g_{t-s}(\yb-\xb) \d \yb,
\]
where $\tilde{\Theta}_{t}(\xb)=\Theta_t(\sumel \xb)$, for $0\leq t < 1$.
To see that $\{\tilde{\Theta}_{t}\}_{0\leq t < 1}$ is harmonic to $\{\mathbf{X}\}_{0\leq t <1}$, note that
\begin{align}
\int_{\R^n}g_{t-s}(\yb-\xb)\tilde{\Theta}_{t}(\yb) \d\mathbf{y}
&=\int_{\R^n}\prod_{i=1}^{n}f_{(t-s)m_i}(y_i-x_i)\tilde{\Theta}_{t}(\yb) \d\mathbf{y} \notag \\ 
&=\int_{\R}\int_{\R^n}f_{\sumel \mb (1-t)}(z-\sumel\yb)\prod_{i=1}^{n}f_{(t-s)m_i}(y_i-x_i) \d\mathbf{y} \frac{\d\nu(z)}{f_{\sumel \mb}(z)} \notag \\ 
&=\int_{\R}f_{\sumel \mb (1-s)}(z-\sumel\xb) \frac{\d\nu(z)}{f_{\sumel\mb}(z)} \label{eq:multi_conv} \\ 
&=\tilde{\Theta}_{s}(\xb) \notag
\end{align}
for $0\leq s < t < 1$, where (\ref{eq:multi_conv}) follows from repeated use of the convolution property of $\{f_t\}_{0\leq t\leq \sumel \mb}$.
\end{rem}

Remark \ref{harmonic} demonstrates that the laws of $\{\xib_t\}_{0\leq t <1}$ and $\{\mathbf{X}_t\}_{0\leq t <1}$ are equivalent, which we formalise by the corollary below.
\begin{coro}
\label{radonnikodymglp}
	Suppose that $\{\xib_{t}\}_{t\geq0}$ is a \levy process under measure $\widetilde{\Q}$ with $\widetilde{\Q}(\xib_{t}\in \d \mathbf{x})=g_t(\xb)\d\xb$.
	Then $\{\Theta_t(R_t)^{-1}\}_{0\leq t <1}$ is a Radon-Nikodym density process that defines the measure change
	\begin{equation}
		\label{eq:RNDPlevy}
		\left.\frac{\dd \widetilde{\Q}}{\dd \Q}\right|_{\F_t^{\xib}}=\Theta_t(R_t)^{-1}, \qquad (0\leq t <1),
	\end{equation} 
	and $\{\xib_t\}_{0\leq t < 1}$ is a $\Q$-GLP with generating law $\nu$ and activity parameter $\mb$.
\end{coro}
\begin{proof}
See Appendix.
\end{proof}

\begin{rem}
\label{condincrementdistgenlaw}
Let $\nu_{st}(B)=\Q(R_t\in B\,|\, \F_s^{\xib})$ for $0\leq s < t \leq 1$ and $B\in\Borel$.
Given $\xib_s$, the increment $\xib_t - \xib_s$ has a generalised multivariate Liouville distribution with generating law $\nu_{st}(B+R_s)$ for $B\in\Borel$ and parameter vector $\mathbf{m}(t-s)$.
\end{rem}

\begin{prop}
Given $\xib_1$, the process $\{\xib_t\}_{0\leq t \leq 1}$ is a vector of independent \levy bridges.
\end{prop}
\begin{proof}
For all $s\in[0,1)$ the transition probabilities to $\xib_t$ ($s<t<1$) can be computed from (\ref{eq:GLP2}) by first substituting $\nu$ with the Dirac measure $\delta_{\sumel \mathbf{z}}$ in (\ref{unormmeasures}), yielding
\begin{equation*}
\Q(\xib_t \in \dd \mathbf{y} \,|\, \xib_1=\mathbf{z}, \F_s^{\xib}) = \prod_{i=1}^n 	\frac{f_{m_i(t-s)}(y_i-\xi^{(i)}_s)f_{m_i(1-t)}(z_i-y_i)}{f_{m_i(1-s)}(z_i-\xi^{(i)}_s)} \d y_i,
\end{equation*}
for almost every $\mathbf{z}\in\R^n$.
Conditional on $\xib_1=\mathbf{z}$, we see that the transition laws of the coordinates of $\{\xib_t\}_{0\leq t \leq 1}$ are independent, and that each is the transition law of a \levy bridge.
\end{proof}
Using the Markov property of $\{\xib_t\}_{0\leq t \leq 1}$, we can also provide the conditional laws of the coordinates $\xi^{(i)}_t$ given $\F_s^{\xi^{(i)}}=\sigma(\{\xi^{(i)}_u\}: 0\leq u \leq s)$ or given $\F_s^{\xib}$, for $s<t$.
\begin{prop}
\label{prop:marginaltransition}
The coordinates of $\{\xib_t\}_{0\leq t < 1}$ have the following transition laws:
\begin{enumerate}
\item The marginally conditioned case:
\begin{align*}
\Q(\xi^{(i)}_t \in \d y_i | \xi^{(i)}_s=x_i) = \frac{\Psi_{t}^{(i)}(y_i)}{\Psi_{s}^{(i)}(x_i)}f_{(t-s)m_i}(y_i-x_i)\d y_i,
\end{align*}
where
\begin{align*}
\Psi_{t}^{(i)}(x)=\int_{\R} \frac{f_{\sumel \mathbf{m} -tm_i}(r-x)}{f_{\sumel \mathbf{m}}(r)}\nu(\dd r). 
\end{align*}
\item The fully conditioned case:
\begin{align*}
\Q\left(\left. \xi^{(i)}_t \in \dd y_i \,\right| \xib_s = \mathbf{x} \right) = \frac{\Theta_t^{(i)}(\mathbf{x},y_i)}{\Theta_s(\sumel \xb)}f_{(t-s)m_i}(y_i-x_i) \d y_i,
\end{align*} 
where
\begin{align*}
\Theta_t^{(i)}(\xb,y)= \int_{\R} \frac{f_{\sumel \mathbf{m}(1-s) + (t-s)m_i}(r-\sumel \xb + (y-x_i))}{f_{\sumel \mathbf{m}}(r)} \, \nu(\dd r),
\end{align*}
\end{enumerate}
for $0 \leq s < t < 1$. 
\end{prop}
\begin{proof}
See Appendix.
\end{proof}

\begin{prop} 
\label{prop:R_GLP}
The process $\{R_t\}_{0\leq t \leq 1}$ is an LRB with generating law $\nu$ and the transition law
\begin{align}
\Q(R_1\in \dd r\,|\, \xib_s=\mathbf{x}) &= \frac{\theta_{s}(\dd r;\sumel \xb)}{\Theta_{s}(\sumel \xb)}, \label{rcondone} \\
\Q(R_t\in \dd r\,|\, \xib_s=\mathbf{x})&= \frac{\Theta_{t}(r)}{\Theta_{s}(\sumel \xb)}f_{(t-s)\sumel \mathbf{m}}(r-\sumel \xb)\dd r. \label{rcondtwo}
\end{align}
\end{prop}
\begin{proof}
See Appendix.
\end{proof}
The next statement is a key result for defining stochastic integrals of integrable LRBs, and hence integrable GLPs.
\begin{prop}
\label{canonicalrep}
If $\E(|R_t|)<\infty$ for all $t\in[0,1]$, then $\{R_t\}_{0\leq t < 1}$ admits the canonical semimartingale representation
\begin{align}
\label{GLPsumcanonical}
R_t = \int_{0}^{t}\frac{\E( R_{1} \,|\, \F^{\xib}_{s}) - R_{s}}{1-s}\dd s + M_t,
\end{align}
for $0\leq t < 1$, where $\{M_t\}_{0\leq t < 1}$ is an $(\F_t^{\xib},\Q)$-martingale with initial state $M_0=0$.
\end{prop}
\begin{proof}
From Proposition \ref{prop:R_GLP}, $\{R_t\}_{0\leq t \leq 1}$ is an LRB. 
Hence, if $\E(|R_t|)<\infty$ for all $t\in(0,1]$, then
\begin{align}
\label{conditionalexpsglp}
\E(R_t \,|\, \xib_{s} = \xb) = \frac{1-t}{1-s}\sumel \xb + \frac{t-s}{1-s}\E(R_1 \,|\, \xib_{s} = \xb ), \hspace{0.1in} s\in[0,t).
\end{align}
We shall use (\ref{conditionalexpsglp}) to prove that $\{M_t\}_{0\leq t < 1}$ given in (\ref{GLPsumcanonical}) is an $\F_t^{\xib}$-martingale. Since $\{\xib_t\}_{0\leq t \leq 1}$ is Markov,
\begin{align*}
\E(M_t - M_s \,|\, \F_s^{\xib}) &= \E(R_t - R_s \,|\, \F_s^{\xib}) - \int_s^t \frac{\E(R_1 \,|\, \xib_s) - \E(R_u \,|\, \xib_s)}{1-u} \dd u \notag \\
&= \frac{1-t}{1-s}\sumel \xib_{s} + \frac{t-s}{1-s}\E(R_1 \,|\, \xib_{s}) - R_{s} - \int_s^t \frac{\E(R_1 \,|\, \xib_s)}{1-u} \dd u \notag \\
&\qquad + \int_s^t \frac{1}{1-u}\left(\frac{1-u}{1-s}\sumel \xib_{s} + \frac{u-s}{1-s}\E(R_1 \,|\, \xib_s)\right) \dd u  \notag \\
&=0,
\end{align*}
for $0\leq s <t \leq 1$.
Given $\E(|R_t|)<\infty$, $\E\left(\int_0^t (|\E(R_1 \,|\, \xib_s) - R_s|)/(1-s) \dd s\right) <\infty$ for $0\leq t < 1$ remains to be shown:
\begin{align*}
\E\left(\int_0^t \frac{|\E(R_1 \,|\, \xib_s) - R_s|}{1-s} \dd s\right) &\leq \E\left(\int_0^t \frac{|\E(R_1 \,|\, \xib_s)|}{1-s} \dd s\right) + \E\left(\int_0^t \frac{|R_s|}{1-s} \dd s\right) \notag \\
&= \int_0^t \E\left(\frac{|\E(R_1 \,|\, \xib_s)|}{1-s}\right) \dd s + \int_0^t \E\left(\frac{|R_s|}{1-s} \right)\dd s \notag \\
&< \infty,
\end{align*}
since $\{\E(R_1 \,|\, \F^{\xib}_t)\}_{0\leq t < 1}$ is a martingale.
Hence, $\E(|M_t|)<\infty$ for $0\leq t < 1$.
Finally, $M_0=0$ since $R_0=0$.
\end{proof}
\begin{rem}
Let $\alpha_t = (1-t)^{-1}$ and $\beta_t = \E(R_{1} \,|\, \xib_{t})$.
Then
\begin{align*}
\dd R_t = \alpha_t\left(\beta_t - R_t \right)\dd t + \dd M_t,
\end{align*}
for $0\leq t < 1$.
In this form, the dynamics of LRBs resemble those of an Ornstein-Uhlenbeck process, with an increasing mean-reversion rate $\{\alpha_t\}_{0\leq t < 1}$ and a state-dependent reversion level $\{\beta_t\}_{0\leq t < 1}$.
We can write
\begin{align*}
R_t = \int_{0}^{t}\frac{1-t}{(1-s)^2}\E( R_{1} \,|\, \xib_{s})\dd s + \int_{0}^{t} \frac{1-t}{1-s}\dd M_s, \hspace{0.25in} \text{for $0\leq s < t < 1$}.
\end{align*}
\end{rem}

The following two propositions are motivated by \cite{7}.
We first recall that a measurable process $\{H_t\}_{t\geq0}$ is called a \emph{harness}, if for all $t\geq 0$, $\E(|H_t|)<\infty$ and for all $0\leq a<b<c<d$, 
\begin{align*}
\E\left(\left. \frac{H_c - H_b}{c-b}  \,\right| \mathcal{H}_{a,d} \right) = \frac{H_d - H_a}{d-a},
\end{align*}
where $\mathcal{H}_{a,d}= \sigma(\{H_t\}_{t\leq a}, \{H_t\}_{t\geq d})$.
\begin{prop}
	\label{harnessprop}
	If $\E(\sumel \xib_t)<\infty$ for $t\in [0,1]$, then $\{\xib_t\}_{0\leq t \leq 1}$ and $\{R_t\}_{0\leq t \leq 1}$ are harnesses.
\end{prop}
\begin{proof}
See Appendix.
\end{proof}
\begin{prop}
\label{enlargedfiltrationprop}
Let $\varphi$ be a $C^1$-function. If $\E(|R_t|)<\infty$ for all $t\in(0,1]$, then the stochastic process $\{Z_t\}_{0\leq t < 1}$ defined by
\begin{align*}
Z_{t} =  \frac{\E( R_{1} \,|\, \xib_t) - R_t}{1-t}\int_t^1 \varphi(u)\d u + \int_0^t \varphi(u)\d R_u, \qquad (0\leq t <1),
\end{align*}
is an $(\mathcal{F}_{t}^{\xib},\Q)$-martingale.
\end{prop}
\begin{proof}
We have
\begin{align*}
\E\left(\left. \int_t^1 \varphi(u)\dd R_u \,\right| \mathcal{F}_t^{\xib} \right) &=\varphi(1)\,\E( R_{1} \,|\, \mathcal{F}_t^{\xib}) - \varphi(t)R_t - \int_t^1\E(R_u \,|\, \mathcal{F}_t^{\xib})\d \varphi(u) \notag \\
&= \frac{\E( R_{1} \,|\, \xib_t) - R_t}{1-t}\int_t^1 \varphi(u)\dd u,
\end{align*}
from the integration by parts formula, (\ref{conditionalexpsglp}) and the Markov property of $\{\xib_t\}_{0\leq t \leq 1}$. Hence, $Z_t = \E(\int_0^1 \varphi(u)\d R_u \,|\, \mathcal{F}_t^{\xib})$, which is an $(\mathcal{F}_{t}^{\xib},\Q)$-martingale.
\end{proof}

Similar to Proposition \ref{canonicalrep}, we have the following result (we omit the proof to avoid repetition):
\begin{prop} \label{GLPcanonicalcoro}
If $\E(|\xib_t|)<\infty$ for all $t\in(0,1]$, then $\{\xib_t\}_{0\leq t < 1}$ admits the canonical semimartingale representation
\begin{align}
\label{GLPcanonicalsemimartingale}
\xib_t = \int_{0}^{t}\frac{\E( \xib_{1} \,|\, \F^{\xib}_{s}) - \xib_{s}}{1-s}\d s + \mathbf{M}_t, \qquad (0\leq t < 1),
\end{align}
where $\{\mathbf{M}_t\}_{0\leq t < 1}$ is an $(\F_t^{\xib},\Q)$-martingale.
\end{prop}

In \cite{16}, it is shown that Archimedean survival processes satisfy the weak Markov consistency condition, but not necessarily the strong Markov consistency condition. 
Motivated by this, Proposition \ref{weakstrong} below provides a generalised version of this result for GLPs. 
First, we recall the weak and strong Markov consistency conditions.
Let $\{\mathbf{X}_t\}_{t\geq0}$ be an $n$-dimensional real-valued Markov process and $\F_t^{\mathbf{X}}=\sigma(\{\mathbf{X}_u\}:0\leq u \leq t)$. 
Also, for each coordinate process $\{X^{(i)}_t\}_{t\geq 0}$, $i=1,\ldots,n$, write $\F_t^{X^{(i)}}=\sigma(\{X^{(i)}_u\}:0\leq u \leq t)\subset \F_t^{\mathbf{X}}$. 
The process $\{\mathbf{X}_t\}_{t\geq0}$ satisfies the weak Markov consistency condition if
\begin{align}
\label{weakmarkovcond}
\Q\left(\left. X^{(i)}_t \in B \,\right| \F_s^{X^{(i)}} \right) = \Q\left(\left. X^{(i)}_t \in B \,\right| X^{(i)}_s \right),
\end{align} 
for every $i=1,\ldots,n$ and every $B\in\mathcal{B}(\mathbb{R})$. Further, $\{\mathbf{X}_t\}_{t\geq 0}$ satisfies the strong Markov consistency condition if
\begin{align}
\label{strongmarkovcond}
\Q\left(\left. X^{(i)}_t \in B \,\right| \F_s^{\mathbf{X}} \right) = \Q\left(\left. X^{(i)}_t \in B \,\right| X^{(i)}_s \right),
\end{align} 
for every $i=1,\ldots,n$ and every $B\in\mathcal{B}(\mathbb{R})$. 

\begin{prop}
\label{weakstrong}
Any GLP $\{\xib_{t}\}_{0\leq t \leq 1}$ is weak Markov consistent, but not necessarily strong Markov consistent. 
\end{prop}
\begin{proof}
Each coordinate $\{\xi^{(i)}_t\}_{0\leq t \leq 1}$ of the GLP $\{\xib_t\}_{0\leq t \leq 1}$ is an LRB, since every subprocess of an LRB is an LRB (see \cite{13}). Thus, (\ref{weakmarkovcond}) is satisfied for every $i=1,\ldots,n$, every $B\in\mathcal{B}(\mathbb{R})$ and all $0\leq s < t \leq 1$.
However, (\ref{strongmarkovcond}) does not necessarily hold since $\Q\left(\left. \xi^{(i)}_t -\xi^{(i)}_s  \in \dd y \,\right| \F_s^{\xib} \right) = \Q\left(\left. \xi^{(i)}_t -\xi^{(i)}_s  \in \dd y \,\right| \sum_j \xi^{(j)}_s  \right)$ is only equal to $\Q\left(\left. \xi^{(i)}_t -\xi^{(i)}_s  \in \dd y \,\right| \xi^{(i)}_s \right)$ if both $\sum_j \xi^{(j)}_s$ and $\xi^{(i)}_s$ are independent from the increment $\xi^{(i)}_t -\xi^{(i)}_s$ for all $0\leq s < t \leq 1$. In such a case the coordinates of $\{\xib_t\}_{0\leq t \leq 1}$ are independent \levy processes.
\end{proof}

In the same spirit, we shall introduce \textit{weak and strong semimartingale consistency} conditions. Definition \ref{weakstrongsemimartingaledefn} below goes beyond a Markov setting, but in the context of GLPs, it offers links to Markov consistency.
\begin{defn}
\label{weakstrongsemimartingaledefn}
Let $\{\sib_t\}_{t\geq 0}$ be an $(\mathcal{F}_t^{\sib},\Q)$-semimartingale, where $\mathcal{F}_t^{\sib}=\sigma(\{\sib_u\}:0\leq u\leq t)$. Let $\{S_t^{(i)}\}_{t\geq0}$ be a coordinate of $\{\sib_t\}_{t\geq0}$, and $\mathcal{F}_t^{S^{(i)}}=\sigma(\{S^{(i)}_u\}:0\leq u\leq t)$, for $i=1,\ldots,n$.
\begin{enumerate}
	\item If $\{S_t^{(i)}\}_{t\geq0}$ admits a decomposition $S_t^{(i)} = a_t^{(i)} + m_t^{(i)}$, where $\{a_t^{(i)}\}_{t\geq0}$ is a \cadlag $\{\mathcal{F}_t^{S^{(i)}}\}$-adapted process with bounded variation and $\{m_t^{(i)}\}_{t\geq0}$ is an $(\mathcal{F}_t^{S^{(i)}},\Q)$-local martingale, then $\{\sib_t\}_{t\geq0}$ is \emph{weakly semimartingale consistent} with respect to $\{S_t^{(i)}\}_{t\geq0}$. 
	If this holds for every $i=1,\ldots,n$, then $\{\sib_t\}_{t\geq0}$ satisfies the \emph{weak semimartingale consistency} condition.
	
	\item Let $\{\sib_t\}_{t\geq0}$ be decomposed as $\sib_t= \boldsymbol{A}_t + \boldsymbol{M}_t$, where $\{\boldsymbol{A}_t\}_{t\geq0}$ is a \cadlag $\{\mathcal{F}_t^{\sib}\}$-adapted process with bounded variation and $\{\boldsymbol{M}_t\}_{t\geq0}$ is an $(\mathcal{F}_t^{\sib},\Q)$-local martingale, with coordinates $\{A_t^{(i)}\}_{t\geq0}$ and $\{M_t^{(i)}\}_{t\geq0}$, respectively.
	Given that $S_t^{(i)}=A_t^{(i)} + M_t^{(i)}$, if $\{A_t^{(i)}\}_{t\geq0}$ is $\{\mathcal{F}_t^{S^{(i)}}\}$-adapted and $\{M_t^{(i)}\}_{t\geq0}$ is an $(\mathcal{F}_t^{S^{(i)}},\Q)$-local martingale, then $\{\sib_t\}_{t\geq0}$ is \emph{strongly semimartingale consistent} with respect to $\{S_t^{(i)}\}_{t\geq0}$. 
	If this holds for every $i=1,\ldots,n$, then $\{\sib_t\}_{t\geq0}$ satisfies the \emph{strong semimartingale consistency} condition.
\end{enumerate}
\end{defn}
\begin{prop}
\label{weakstrongsemimartingale}
Any GLP $\{\xib_{t}\}_{0\leq t < 1}$, where $\E(|\xib_t|)<\infty$ for all $t\in(0,1]$, is weak semimartingale consistent, but not necessarily strong semimartingale consistent. 
\end{prop}
\begin{proof}
Let $\E(|\xib_t|)<\infty$ for all $t\in(0,1]$ and define $\alpha_t=(1-t)^{-1}$ for $t<1$.
Following similar steps to the proof of Proposition \ref{canonicalrep}, each coordinate of $\{\xib_{t}\}_{0\leq t < 1}$ admits a decomposition $\xi^{(i)}_t = a^{(i)}_t + m^{(i)}_t$, where
\begin{align*} 
a^{(i)}_t = \int_{0}^{t}\alpha_s\left(\E\left(\left. \xi^{(i)}_1 \,\right| \F^{\xi^{(i)}}_{s}\right) - \xi^{(i)}_s\right)\dd s,
\end{align*} 
which is $\{\mathcal{F}_t^{\xi^{(i)}}\}$-adapted, and $\{m^{(i)}_t\}_{0\leq t < 1}$ is an $(\F_t^{\xi^{(i)}},\Q)$-martingale, for $i=1,\ldots,n$.
Hence, $\{\xib_{t}\}_{0\leq t < 1}$ is weak semimartingale consistent.
From Proposition \ref{GLPcanonicalcoro}, we also know that $\xi^{(i)}_t=A^{(i)}_t+M^{(i)}_t$, where 
\begin{align*}
A^{(i)}_t = \int_{0}^{t}\alpha_s\left(\E\left(\left. \xi^{(i)}_1 \,\right| \F^{\xib}_{s}\right) - \xi^{(i)}_s\right)\dd s,
\end{align*} 
which is $\{\mathcal{F}_t^{\xib}\}$-adapted, and $\{M_t^{(i)}\}_{0\leq t < 1}$ is an $(\mathcal{F}_t^{\xib},\Q)$-martingale. Since $\{\xib_{t}\}_{0\leq t \leq 1}$ is Markov and using Proposition \ref{weakstrong}, we know that $\E( \xi^{(i)}_1 \,|\, \xib_{t})$ is not necessarily equal to $\E( \xi^{(i)}_1 \,|\, \xi^{(i)}_{t})$.
Hence, $\{A^{(i)}_t\}_{0\leq t < 1}$ is not necessarily $\{\mathcal{F}_t^{\xi^{(i)}}\}$-adapted. Also, $\{M_t^{(i)}\}_{0\leq t < 1}$ is not necessarily an $(\mathcal{F}_t^{\xi^{(i)}},\Q)$-martingale. 
\end{proof}
We used Proposition \ref{weakstrong} to prove Proposition \ref{weakstrongsemimartingale}; we shall note another link between semimartingale consistency and Markov consistency. 
From \cite{17}, if a Markov process $\{\mathbf{X}_t\}_{t\geq0}$ satisfies the weak Markov consistency with respect to its marginal $\{X^{(i)}_t\}_{t\geq0}$, then $\{\mathbf{X}_t\}_{t\geq0}$ is also strongly Markov consistent with respect to $\{X^{(i)}_t\}_{t\geq0}$ if and only if $\{\mathcal{F}_t^{X^{(i)}}\}_{t\geq0}$ is $\Q$-immersed in $\{\mathcal{F}_t^{\mathbf{X}}\}_{t\geq0}$. 
Here, $\Q$-immersion means that if $\{X^{(i)}_t\}_{t\geq0}$ is an $(\mathcal{F}_t^{X^{(i)}},\Q)$-local martingale, then it is an $(\mathcal{F}_t^{\mathbf{X}},\Q)$-local martingale. 
As an opposite direction to $\Q$-immersion, we prove a result that links strong martingale consistency and strong Markov consistency.
\begin{prop}
\label{strongtostrong}
Let $\{\sib_t\}_{t\geq 0}$ be a Markov $(\mathcal{F}_t^{\sib},\Q)$-martingale, satisfying weak Markov consistency. 
Then, $\{\sib_t\}_{t\geq0}$ is strong semimartingale consistent if and only if it is strong Markov consistent.
\end{prop}
\begin{proof}
Since $\{\sib_t\}_{t\geq 0}$ is an $(\mathcal{F}_t^{\sib},\Q)$-martingale, we have $\E( S^{(i)}_t \,|\, \F^{\sib}_{u})=S^{(i)}_u$ for $0\leq u < t$. 
Then, if $\{\sib_t\}_{t\geq0}$ is strong martingale consistent, we have $\E( S^{(i)}_t \,|\, \F^{\sib}_{u})=\E( S^{(i)}_t \,|\, \mathcal{F}_u^{S^{(i)}})=S^{(i)}_u$. 
Thus, given that $\{\sib_t\}_{t\geq 0}$ is Markovian satisfying weak Markov consistency,
\begin{multline*}
\E( S^{(i)}_t \,|\, \F^{\sib}_{u}) = \int_{\R}x\P\left(\left. S^{(i)}_t \in \dd x \,\right| \F^{\sib}_{u} \right)=\int_{\R}x\P\left(\left. S^{(i)}_t \in \dd x \,\right| \sib_u \right) \\
=\E( S^{(i)}_t \,|\, \mathcal{F}_u^{S^{(i)}})= \int_{\R}x\P\left(\left. S^{(i)}_t \in \dd x \,\right| \mathcal{F}_u^{S^{(i)}} \right)=\int_{\R}x\P\left(\left. S^{(i)}_t \in \dd x \,\right| S^{(i)}_u \right).
\end{multline*}
For the opposite direction, if $\{\sib_t\}_{t\geq0}$ is strong Markov consistent, then since $\{\sib_t\}_{t\geq 0}$ is an $(\mathcal{F}_t^{\sib},\Q)$-martingale satisfying weak Markov consistency,
\begin{multline*}
\int_{\R}x\P\left(\left. S^{(i)}_t \in \dd x \,\right| \F^{\sib}_{u} \right)=\E( S^{(i)}_t \,|\, \F^{\sib}_{u}) =S^{(i)}_u \\
=\int_{\R}x\P\left(\left. S^{(i)}_t \in \dd x \,\right| S^{(i)}_u \right) = \int_{\R}x\P\left(\left. S^{(i)}_t \in \dd x \,\right| \mathcal{F}_u^{S^{(i)}} \right) = \E( S^{(i)}_t \,|\, \mathcal{F}_u^{S^{(i)}}).
\end{multline*} 
Hence, $\{S_t^{(i)}\}_{t\geq0}$ is also an $(\mathcal{F}_t^{S^{(i)}},\Q)$- martingale, and the statement follows.
\end{proof}

\section{Examples}
We shall now study two examples of GLPs in more detail: Brownian Liouville processes and Poisson Liouville processes.

\subsection{Brownian Liouville processes}
As a subclass of GLPs, let us consider what we call Brownian Liouville processes (BLPs). 
In Definition \ref{def:GLP}, we let $\{L_t\}_{0\leq t \leq u_n}$ be a Brownian random bridge given by
\begin{align}
\label{BRBCanonical}
L_t = \frac{t}{\sumel \mathbf{m}}L_{\sumel \mathbf{m}}+\s \left(W_t-\frac{t}{\sumel \mathbf{m}}W_{\sumel \mathbf{m}}\right),
\end{align}
where $\sigma>0$ and $\{W_t\}_{0\leq t \leq u_n}$ is a standard Brownian motion independent of the random variable $L_{\sumel \mathbf{m}}$. For a background of the anticipative orthogonal representation given in (\ref{BRBCanonical}) for a Brownian random bridge, we refer the reader to \cite{8} and \cite{10}. We also note that the Gaussian process  $\{W_t-\frac{t}{\sumel \mathbf{m}}W_{\sumel \mathbf{m}}\}_{0\leq t \leq \sumel \mathbf{m}}$ in (\ref{BRBCanonical}) is a Brownian bridge starting and ending at zero.
The following proposition is analogous to \cite[Proposition 3.10]{14} for Archimedean survival processes.
We denote the Hadamard (i.e.~element-wise) product of vectors $\xb, \yb \in \R^n$ by $\xb \circ \yb$.
We say $\{\beta_t\}_{0\leq t \leq 1}$ is a \emph{standard Brownian bridge} if (a) it is a Brownian bridge, (b) $\beta_0=\beta_1=0$, and (c) $\var(\beta_t)=(1-t)^2$.
\begin{prop} \label{BLPAnticipative}
If $\{\xib_t\}_{0\leq t \leq 1}$ is a BLP with the master process (\ref{BRBCanonical}), then it admits the independent Brownian bridge representation:
\begin{align}
\label{blpindependentrep}
\boldsymbol{\xi}_t = t\left(\frac{\mathbf{m}}{\sumel \mathbf{m}} R_1 + \sigma \mathbf{Z} \right) + \sigma \sqrt{\mathbf{m}} \circ \boldsymbol{\b}_t,
\end{align}
where $\sqrt{\mathbf{m}}=(\sqrt{m_1},\ldots,\sqrt{m_n})^\tp$, $\{\boldsymbol{\b}_t\}$ is a vector of independent standard Brownian bridges and the random vector $\mathbf{Z} = (Z_1,\ldots,Z_n)^{\tp}$ is multivariate Gaussian with
\begin{align*}
 \cov(Z_i, Z_j) = \delta_{ij}m_i - \frac{m_im_j}{\sumel \mathbf{m}}.
\end{align*}
\end{prop}
\begin{proof}
For the proof, we use $W_t$ and $W(t)$ interchangeably. We have
\begin{align*}
	\xi^{(i)}_t &= L(m_it-u_{i-1}) - L(u_{i-1}) \notag \\
	&=\frac{m_it}{\sumel \mathbf{m}}L(\sumel \mathbf{m}) + \sigma\left( W(m_it+u_{i-1}) - W(u_{i-1})-\frac{m_it}{\sumel \mathbf{m}}W(\sumel \mathbf{m}) \right) \notag \\
	&=\frac{m_it}{\sumel \mathbf{m}}R_1 + \sigma\sqrt{m_i}\b^{(i)}_t + \sigma t\left( W(u_{i}) - W(u_{i-1})-\frac{m_i}{\sumel \mathbf{m}}W(\sumel \mathbf{m}) \right),
\end{align*}
since $R_1=L_{\sumel \mathbf{m}}$, and where
\begin{align*}
\sqrt{m_i}\b^{(i)}_t = W(m_it + u_{i-1}) - W(u_{i-1}) - t(W(u_i) - W(u_{i-1})).
\end{align*}
So $\{\b^{(i)}_t\}_{0\leq t \leq 1}$ is a standard Brownian bridge, and is independent of $W(u_i) - W(u_{i-1})$ and $W(\sumel \mathbf{m})$. It is straightforward to verify the independence by noting that they are jointly Gaussian with nil covariation.
Thus we can write
\begin{align*}
\xi^{(i)}_t = t \left(\frac{m_i}{\sumel \mathbf{m}}R_1 + \sigma Z_i\right)+ \sigma\sqrt{m_i}\b^{(i)}_t,
\end{align*}
where $Z_i$ is given by
\begin{align}
\label{indgaussianrv}
Z_i =  W(u_{i}) - W(u_{i-1})-\frac{m_i}{\sumel \mathbf{m}}W(\sumel \mathbf{m}),
\end{align}
and $R_1$, $Z_i$ and $\{\b^{(i)}_t\}_{0\leq t \leq 1}$ are mutually independent.
Furthermore, $\{\b^{(1)}_t\}_{0\leq t \leq 1}$,\ldots, $\{\b^{(n)}_t\}_{0\leq t \leq 1}$ are mutually independent, since they are jointly Gaussian with nil covariation.
\end{proof}
Proposition \ref{BLPAnticipative} provides an anticipative orthogonal representation for BLPs, whereas (\ref{GLPcanonicalsemimartingale}) provides a non-anticipative semimartingale representation when $\{\xib_t\}_{0\leq t \leq 1}$ is a BLP.
\begin{rem}
\label{remforZ}
Note that $\sumel \mathbf{Z}=0$ from (\ref{indgaussianrv}), and so its covariance matrix is singular. Also, using (\ref{blpindependentrep}) from Proposition \ref{BLPAnticipative}, $\{R_t\}_{0\leq t \leq 1}$ admits the anticipative representation
\begin{align*}
R_t &= \sumel \xib_t = t\sumel\left(\frac{\mathbf{m}}{\sumel \mathbf{m}} R_1 + \sigma \mathbf{Z}\right) + \sigma\sumel(\sqrt{\mathbf{m}} \circ \boldsymbol{\b}_t) \notag \\
&= tR_1 + \sigma\sqrt{\sumel \mathbf{m}}\tilde{\b}_t,
\end{align*}
where $\{\tilde{\b}_t\}_{0\leq t \leq 1}$ is a standard Brownian bridge.
\end{rem}
\begin{prop} \label{measureSDEAll}
Let $\pi_0(\dd \textbf{x}) = \Q(\xib_{1} \in \dd \textbf{x})$ and $\pi_t(\dd \textbf{x}) = \Q(\xib_{1} \in \dd \textbf{x} \,|\, \mathcal{F}_t^{\xib})$. Also let $\{\textbf{B}_t\}_{0\leq t < 1}$ be a vector of standard $(\F_t^{\xib},\Q)$-Brownian motions. Then, if $\E(|\xib_{1}|)<\infty$, the multivariate measure-valued process $\{\pi_t\}_{0\leq t < 1}$ satisfies
\begin{align*}
\pi_t(\dd \textbf{x}) &= \pi_0(\dd \textbf{x}) + \int_{0}^t \pi_s(\dd \textbf{x}) \sigmab_s^\tp\left(\sigma\sqrt{\mathbf{m}}\circ\d \textbf{B}_s\right),
\end{align*}
for $0\leq t < 1$, where each coordinate of $\{\sigmab_t\}_{0\leq t < 1}$ is given by
\begin{align*}
\sigma_t^{(i)} &= \frac{x^{(i)}-\E(\xi_1^{(i)}\,|\, \xib_t)}{\sigma^2 m_i (1-t)}.
\end{align*}
\end{prop}
\begin{proof}
See Appendix.
\end{proof}

\begin{rem}
Note that $\sumel\tilde{\textbf{B}}_t=\sumel(\sigma \sqrt{\mathbf{m}}\circ\textbf{B}_t)$ gives the non-anticipative semimartingale representation for $\{R_t\}_{0\leq t < 1}$, which is
\begin{align*}
R_t = \int_{0}^{t}\frac{\E( R_{1} \,|\, \xib_{s}) - R_{s}}{1-s}\dd s + \sigma\sum_{i=1}^n \sqrt{m_i}B_t^{(i)},
\end{align*}
which provides the explicit example of the $(\mathcal{F}_t^{\xib},\Q)$-martingale in (\ref{GLPsumcanonical}). 
\end{rem}

\subsection{Poisson Liouville processes}
Our second example are Poisson Liouville processes (PLPs), which are counting processes.
Accordingly, in Definition \ref{def:GLP}, we let $\{L_t\}_{0\leq t \leq u_n}$ be a Poisson random bridge with $\P(L_{u_n}=i)=\nu(\{i\})$, for $i\in\N_0$.
\begin{prop}
\label{intensitynormprocess}
Let $\{\lambda_t^R\}_{0\leq t \leq 1}$ be the intensity process of the $L_1$-norm process $\{R_t\}_{0\leq t \leq 1}$.
If $\E(R_1)<\infty$, then
\begin{align}
\label{plpnormintensity}
\lambda^R_t = \frac{\E(R_1 \,|\, \xib_{t}) - R_t}{1-t}, \qquad \text{for $0\leq t < 1$}.
\end{align}
\end{prop}
\begin{proof}
Since $\{R_t\}_{0\leq t \leq 1}$ is a counting process, we have $\lambda^R_t = \lim_{h\rightarrow 0}\E(R_{t+h}  - R_{t} \,|\, \F_t^{\xib})/h$. Since $\{R_t\}_{0\leq t \leq 1}$ is a Markov process with respect to $\{\F_t^{\xib}\}_{0\leq t \leq 1}$, using (\ref{conditionalexpsglp}), we have
\begin{align*}
\lambda^R_t &= \lim_{h\rightarrow 0} \left(\frac{\E(R_{t+h} \,|\, \xib_t)}{h} -\frac{R_t}{h}\right) \notag \\
&= \lim_{h\rightarrow 0} \left(\frac{1-t-h}{(1-t)h}\sumel \xib_t + \frac{h\E(R_{1} \,|\, \xib_t)}{(1-t)h} -\frac{R_t}{h}\right) \notag\\
&= \frac{\E(R_{1} \,|\, \xib_t)}{(1-t)} + \lim_{h\rightarrow 0} \left(R_t\left(\frac{1-t-h}{(1-t)h} -\frac{1}{h}\right)\right)  \notag \\
&= \frac{\E(R_{1} \,|\, \xib_t)}{1-t} - \lim_{h\rightarrow 0} \left(R_t\frac{h}{(1-t)h}\right),
\end{align*}
which yields the result.
\end{proof}
\begin{rem}
When $\{\xib_t\}_{0\leq t \leq 1}$ is a PLP, Proposition \ref{intensitynormprocess} provides an alternative proof for Proposition \ref{canonicalrep}, since $\{R_t\}_{0\leq t \leq 1}$ is a counting process. 
\end{rem}
\begin{prop}
\label{intensitymargprocess}
Let $\{\lambda_t^{(i)}\}_{0\leq t \leq 1}$ be the intensity process of the coordinate $\{\xi^{(i)}_t\}_{0\leq t \leq 1}$. If $\E(R_1)<\infty$, then
\begin{align*}
\lambda^{(i)}_t = \frac{m_i}{\sumel \mathbf{m}}\lambda^R_t, \qquad \text{for $0\leq t < 1$.}
\end{align*}
\end{prop}
\begin{proof}
Fix $0\leq s < t < 1$. From Remark \ref{condincrementdistgenlaw}, we know that given $\xib_s$ the increment $\xib_t - \xib_s$ has a generalised multivariate Liouville distribution with generating law $\nu^*(\{i\})=\nu_{st}(\{i+R_s\})$, for $i\in\N_0$ and parameter vector $\mathbf{m}(t-s)$.
We define
\begin{align*}
\mu_{st} &= \sum_{i=0}^\infty i\, \nu^*(\{i\})  = \sum_{i=R_s}^{\infty} i \, \nu_{st}(\{i\}) - R_s = \E(R_t \,|\, \xib_{s}) - R_s \notag \\
&=  \frac{t-s}{1-s}\left(\E(R_1 \,|\, R_{s}) - R_s\right),
\end{align*}
where the last equality comes from (\ref{conditionalexpsglp}). Then from \cite{2} (Theorem 6.3) and \cite{6}, we have
\begin{align*}
\E(\xi^{(i)}_t \,|\, \xib_{s}) = \frac{m_i}{\sumel \mathbf{m}}\mu_{st} + \xi^{(i)}_s.
\end{align*}
Since $\{\xi^{(i)}_t\}_{0\leq t \leq 1}$ is a counting process, we have $\lambda^{(i)}_t = \lim_{h\rightarrow 0}\E[\xi^{(i)}_{t+h}  - \xi^{(i)}_{t} \,|\, \F_t^{\xib}]/h$. Thus,
\begin{align*}
\lambda^{(i)}_t &= \frac{m_i}{\sumel \mathbf{m}}\lim_{h\rightarrow 0}\frac{\mu_{t,t+h}}{h} \notag \\
& = \frac{m_i}{\sumel \mathbf{m}}\frac{\E(R_1 \,|\, R_{t}) - R_t}{1-t}.
\end{align*}
The result then follows from (\ref{plpnormintensity}).
\end{proof}
\begin{rem}
Note that $\lambda^{R}_t = \sum_i\lambda^{(i)}_t$.
\end{rem}

If we let $\{P_t\}_{0\leq t \leq 1}$ denote a Poisson process, and define $\boldsymbol{\D}$ by $\D_i = P_{t_i}-P_{t_{i-1}}$ for some partition $0=t_0<t_1<\cdots<t_n$, we have
\begin{align*}
\Q(\boldsymbol{\D}=\mathbf{x}\,|\, P_{t_n}=k) = \left\{\begin{aligned}
													& k! \prod_{i=1}^n \frac{p_i^{x_i}}{x_i!}, &&\text{for $\xb\in\N_0^n$ with $\sumel \mathbf{x}=k$,}
												\\ &0, &&\text{otherwise,}
												\end{aligned}\right.
\end{align*}
where $p_i=(t_i-t_{i-1})/t_n$.
In other words, given $P_{t_n}$, $\boldsymbol{\D}$ has a multinomial distribution.
Let $\{\xib_t\}_{0\leq t \leq 1}$ be a Poisson Liouville process with generating law $\nu(\{k\})=A(k)$, $k\in\N_0$.
Then
\begin{align*}
\P(\xib_1=\mathbf{x}) = A(\sumel \mathbf{x}) (\sumel \mathbf{x})! \prod_{i=1}^n \frac{p_i^{x_i}}{x_i!}.
\end{align*}
Write $G_\nu$ for the probability generating function of $\nu$:
\begin{align*}
G_{\nu}(z)=\sum_{k=0}^{\infty}z^kA(k).
\end{align*}
Let $T^{(i)}$ be the time of the first jump of the $i$th marginal process.
If $\xi^{(i)}_1<1$, then we set $T^{(i)}=\infty$.
\begin{prop} \label{poiossonmarg}
The random times $\{T^{(i)}; i=1,\ldots,n\}$ satisfy the following:
\begin{align*}
	\P(T^{(i)}>s) &= G_{\nu}(1-s p_i), \notag \\
	\P(T^{(i)}=\infty) &= G_{\nu}(1-p_i), \notag \\
	\P(T^{(i)}> s_i; i=1,\ldots,n)&=G_{\nu}\left( 1-\sum_{i=1}^n p_i s_i \right), \notag \\
	\P(T^{(i)}=\infty; i=1,\ldots,n) &= A(0),
\end{align*}
for $s\in[0,1]$ and $\mathbf{s}\in [0,1]^n$.
\end{prop}
\begin{proof}
See Appendix.
\end{proof}

$G_\nu$ is increasing (and invertible) on $[0,1]$.
Write $\psi(x)=G_\nu(1-x)$, and note that $\psi$ is invertible on $[0,1]$.
If $u_i\in[0,\psi(p_i)]$, then we have
\begin{align*}
\P\left(T^{(i)}> \frac{\psi^{-1}(u_i)}{p_i}\right) = u_i.
\end{align*}
It follows that the conditioned random variable $\psi(p_i T^{(i)}) \,|\, \{T^{(i)} < 1\}$ is uniformly distributed.
Furthermore
\begin{align*}
\P\left(T^{(i)}> \frac{\psi^{-1}(u_i)}{p_i}; i=1,\ldots,n\right) = \psi\left(\sum_i \psi^{-1}(u_i)\right).
\end{align*}
The form of the joint survival function of $\mathbf{T}=\{T^{(i)}; i=1,\ldots,n\}$ resembles that of an Archimedean copula.
However, the fact that $\P(T^{(i)}\geq 1)>0$ means that it is not an Archimedean copula.

\section{Conclusion}
We introduced generalised Liouville processes: a broad and tractable class of multivariate stochastic processes.
The class of GLPs generalises some processes that have already been studied.
We detailed various properties of GLPs and provided some new examples.

\section*{Acknowledgements}
The authors are grateful to an anonymous referee whose careful reading and comments led to significant improvements in this paper.

\appendix

\section{Proofs}

\subsection{Proposition \ref{eq:finitedimensionalGLPincrementcharacterise}}
\begin{proof}
Since $\nu$ has no continuous singular part,
$\nu(\dd z)=\sum^{\infty}_{j=-\infty}c_{i}\delta_{z_{i}}(z)\d z+p(z)\d z$,
where $c_{i}\in\mathbb{R}_+$ is a point mass of $\nu$ located at $z_{i}\in\mathbb{R}$, and $p:\mathbb{R}\rightarrow\mathbb{R}_{+}$ is the density of the continuous part of $\nu$.
Then from (\ref{eq:LRBfinitedimensionaldistributionglp}), the joint density of an LRB $\{L_t\}_{0\leq t \leq u_n}$ is given by
\begin{multline*}
	\Q(L_{t_1} \in \dd x_1, \ldots, L_{t_k} \in \dd x_k, L_{u_n}\in \dd x_n)
	  \\= \prod_{i=1}^n[f_{t_i-t_{i-1}}(x_i-x_{i-1}) \d x_i] \frac{\sum^{\infty}_{j=-\infty}c_{i}\delta_{z_{i}}(x_n)+p(x_n)}{f_{u_n}(x_n)},
\end{multline*}
where $x_0=0$, for all $k\in\N_+$, all partitions $0=t_0<t_1<\cdots<t_{n-1}<t_{n}=u_n$, all $x_{n}\in\R$, and all $(x_1,\ldots,x_k)^\tp=\mathbf{x} \in \R^k$. 
Let $\ab\in\mathbb{R}^{n}_{+}$ be the vector of time increments $\alpha_{i}=t_{i}-t_{i-1}$, and $\alpha=\sumel \ab=u_n$. 
The Jacobian of the transformation $y_{1}=x_{1}, y_{2}=x_{2}-x_{1}, \ldots y_{n}=x_{n}-x_{n-1}$ is unity, and it follows that
\begin{multline*}
	\Q(L_{t_1}-L_{t_0} \in \dd y_1, \ldots, L_{u_n}-L_{t_k}\in \dd y_n) 
	  \\= \prod_{i=1}^nf_{\alpha_{i}}(y_i)\d y_{i} \frac{\sum^{\infty}_{j=-\infty}c_{i}\delta_{z_{i}}(\sum^{n}_{i=1}y_{i})+p(\sum^{n}_{i=1}y_{i})}{f_\alpha(\sum^{n}_{i=1}y_{i})}.
\end{multline*}
From \cite{6}, we know that $(L_{t_1}-L_{t_0},\ldots,L_{t_k}-L_{t_{k-1}},L_{u_n}-L_{t_k})^{\tp}$
has a generalised multivariate Liouville distribution.
Fix $k_i\geq 1$ and the partitions $0=t_0^i<t_1^i<\cdots<t_{k_i}^i=1$, for $i=1,\ldots,n$. 
Define the non-overlapping increments $\{\D_{ij}\}$ by $\D_{ij}=\xi^{(i)}_{t^i_j}-\xi^{(i)}_{t^i_{j-1}}$,
for $j=1,\ldots,k_i$ and $i=1,\ldots,n$. 
The distribution of the $k_{1}\times \cdots \times k_{n}$-element vector
$\boldsymbol{\D}=(\D_{11},\ldots,\D_{1k_1},\ldots, \D_{n1},\ldots,\D_{nk_n})^\tp$
characterises the finite-dimensional distributions of the GLP $\{\xib_t\}_{0\leq t \leq 1}$. 
It follows from the Kolmogorov extension theorem that the distribution of $\boldsymbol{\D}$ characterises the law of $\{\xib_t\}_{0\leq t \leq 1}$. 
Note that $\boldsymbol{\D}$ contains non-overlapping increments of $\{L_t\}$ such that $\sumel \boldsymbol{\D}=L_{u_n}$.
Hence, $\boldsymbol{\D}$ has a generalised multivariate Liouville distribution.
\end{proof}

\subsection{Proposition \ref{markovtransitionprob}}
\begin{proof}
	We compute the transition probabilities of $\{\xib_t\}_{0\leq t \leq 1}$ directly.
	We begin by transitioning to $\xib_t$ for $t<1$.
	For all $k\geq 2$, all $0<t_1<\cdots<t_k<1$ and all $\xb_1,\ldots,\xb_{k} \in \R^n$, we have
	\begin{align}
	&\P(\xib_{t_k}\in \dd \xb_k \,|\, \xib_{t_1}=\xb_1,\ldots, \xib_{t_{k-1}}=\xb_{k-1}) \nonumber \\
	&\qquad= \frac{\P(\xib_{t_1}\in\dd \xb_1,\ldots, \xib_{t_{k}}\in\dd\xb_{k})}
	{\P(\xib_{t_1}\in\dd\xb_1,\ldots, \xib_{t_{k-1}}\in\dd\xb_{k-1})} \nonumber \\
	&\qquad= \frac{\int_{-\infty}^\infty \P(\xib_{t_1}\in \dd \xb_1,\ldots, \xib_{t_k}\in \dd\xb_k \,|\, R_1=z)\nu(\dd z)}
	{\int_{-\infty}^\infty \P(\xib_{t_1}\in\dd\xb_1,\ldots, \xib_{t_{k-1}}\in\dd\xb_{k-1} \,|\, R_1=z) \nu(\dd z)} \nonumber \\
	&\qquad= \frac{\int_{-\infty}^\infty \P(\xib_{t_1}-\xib_{t_0}\in -\xb_{0}+\dd \xb_1,\ldots, \xib_{t_k}-\xib_{t_{k-1}}\in -\xb_{k-1}+\dd\xb_k \,|\, R_1=z)\nu(\dd z)}
	{\int_{-\infty}^\infty \P(\xib_{t_1}-\xib_0\in -\xb_{0}+\dd \xb_1,\ldots, \xib_{t_{k-1}}-\xib_{t_{k-2}}\in -\xb_{k-2}+\dd\xb_{k-1} \,|\, R_1=z) \nu(\dd z)} \label{eq:inc} \\
	&\qquad= \frac{\int_{-\infty}^\infty \left\{ \prod_{i=1}^n\prod_{j=1}^k f_{m_i(t_j-t_{j-1})}(x^{(i)}_j-x^{(i)}_{j-1}) \d x^{(i)}_j \right\} f_{\sumel \mb(1-t_k)}(z-\sumel \xb_k) \frac{\nu(\dd z)}{f_{\sumel \mb}(z)}}
	{\int_{-\infty}^\infty \left\{ \prod_{i=1}^n\prod_{j=1}^{k-1} f_{m_i(t_j-t_{j-1})}(x^{(i)}_j-x^{(i)}_{j-1}) \d x^{(i)}_j \right\} f_{\sumel \mb(1-t_{k-1})}(z-\sumel \xb_k)\frac{\nu(\dd z)}{f_{\sumel \mb}(z)}} \label{eq:cyc} \\
	&\qquad= \frac{\int_{-\infty}^\infty \left\{ \prod_{i=1}^n f_{m_i(t_k-t_{k-1})}(x^{(i)}_k-x^{(i)}_{k-1}) \d x^{(i)}_k \right\} f_{\sumel m(1-t_k)}(z-\sumel \xb_k) \frac{\nu(\dd z)}{f_{\sumel \mb}(z)}}
	{\int_{-\infty}^\infty f_{\sumel m(1-t_{k-1})}(z-\sumel \xb_{k-1})\frac{\nu(\dd z)}{f_{\sumel \mb}(z)}} \nonumber \\
	&\qquad= \frac{\Theta_{t_k}(\sumel \xb_{k})}{\Theta_{t_{k-1}}(\sumel \xb_{k-1})} \prod_{i=1}^n f_{m_i(t_k-t_{k-1})}(x^{(i)}_k-x^{(i)}_{k-1}) \d x^{(i)}_k, \nonumber
	\end{align}
	where $t_0=0$, $\xb_0=\mathbf{0}$ and $x^{(i)}_j$ is the $i$th coordinate of $\xb_j$.
	We provide some remarks on the step (\ref{eq:inc}) to (\ref{eq:cyc}).
	Note that in (\ref{eq:inc}) all the increments of type $\xib_t-\xib_s$ are vectors of non-overlapping increments of the master LRB $\{L_t\}_{0\leq t \leq \sumel \mb}$.
	Given $R_1=L_{\sumel \mb}$, $\{L_t\}_{0\leq t \leq \sumel \mb}$ is a \levy bridge, and so its law is invariant to a reordering of its non-overlapping increments.
	This is a direct result of the so-called cyclical exchangeability property of \levy bridges (see, for example, \cite{18}).
	The integrands in (\ref{eq:cyc}) can then be recognised as \levy bridge transition probabilities.
	
	We now consider transitioning to $\xib_1$.
	For all $k\geq 1$, all $0<t_1<\cdots<t_k<1$, all $\xb_1,\ldots,\xb_k \in \R^n$, all $y_1,\ldots,y_{k-1}\in\R$ and all $B\in\mathcal{B}(\R)$, we have
	\begin{align}
	&\P(\xi^{(1)}_1\in \dd y_1,\ldots,\xi^{(n-1)}_1\in \dd y_{n-1}, \xi^{(n)}_1\in B\,|\,  \xib_{t_1}=\xb_1,\ldots, \xib_{t_{k}}=\xb_{k}) \nonumber \\
	&\qquad=\frac{\P(\xi^{(1)}_1\in \dd y_1,\ldots,\xi^{(n-1)}_1\in \dd y_{n-1}, \xi^{(n)}_1\in B, \xib_{t_1}\in\dd\xb_1,\ldots, \xib_{t_{k}}\in\dd\xb_{k})}
	{\P(\xib_{t_1}\in\dd\xb_1,\ldots, \xib_{t_{k-1}}\in\dd\xb_{k})} \nonumber \\
	&\qquad=\frac{\P(\xi^{(1)}_1\in \dd y_1,\ldots,\xi^{(n-1)}_1\in \dd y_{n-1}, R_1\in B+\sum_{i=1}^{n-1} y_i, \xib_{t_1}\in\dd\xb_1,\ldots, \xib_{t_{k}}\in\dd\xb_{k})}
	{\P(\xib_{t_1}\in\dd\xb_1,\ldots, \xib_{t_{k-1}}\in\dd\xb_{k})} \nonumber \\
	&\qquad=\frac{\int_{z\in B+\sum_{i=1}^{n-1} y_i}\P(\xi^{(1)}_1\in \dd y_1,\ldots,\xi^{(n-1)}_1\in \dd y_{n-1}, \xib_{t_1}\in\dd\xb_1,\ldots, \xib_{t_{k}}\in\dd\xb_{k} \,|\,R_1=z)\nu(\dd z)}
	{\int_{-\infty}^{\infty}\P(\xib_{t_1}\in\dd\xb_1,\ldots, \xib_{t_{k-1}}\in\dd\xb_{k}\,|\,R_1=z)\nu(\dd z)} \nonumber \\
	&\qquad=\frac{\int_{z\in B+\sum_{i=1}^{n-1} y_i}\left\{\prod_{i=1}^{n-1}f_{m_i(1-t_k)}(y_i-x_k^{(i)})\d y_i \right\}f_{m_n(1-t_k)}\left(z-x_k^{(n)}-\sum_{i=1}^{n-1}y_i\right)\frac{\nu(\dd z)}{f_{\sumel \mb}(z)}}
	{\int_{-\infty}^\infty f_{\sumel \mb(1-t_{k-1})}(z-\sumel \xb_{k-1})\frac{\nu(\dd z)}{f_{\sumel \mb}(z)}} \label{eq:cyc2} \\
	&\qquad= \frac{\theta_{\tau(t_k)}\left(B+\sum_{i=1}^{n-1} y_i; x_k^{(n)}+\sum_{i=1}^{n-1}y_i  \right)}{\Theta_{t_{k-1}}(\sumel \xb_{k-1})}\prod_{i=1}^{n-1}f_{m_i(1-t_k)}(y_i-x_k^{(i)})\d y_i, \nonumber
	\end{align}
	where again $t_0=0$ and (\ref{eq:cyc2}) follows from similar arguments to (\ref{eq:cyc}).
\end{proof}

\subsection{Corollary \ref{radonnikodymglp}}
\begin{proof}
The process $\{R_t\}_{0\leq t \leq 1}$ is a \levy process under $\widetilde{\Q}$, where $\widetilde{\mathbb{P}}(R_{t}\in \d x)=f_{t(\sumel\mb)}(x)\d x$. 
To show $\E^{\widetilde{\Q}}\left(|\Theta_t(R_t) |\right)< \infty$,
use the Chapman-Kolmogorov convolution and the non-negativity of $f$:
\begin{align*}
\int_{\R}\left(\int_{\R} \left|\frac{f_{\sumel \mathbf{m}(1-t)}(z-r)f_{t(\sumel \mathbf{m})}(r)}{f_{\sumel \mathbf{m}}(z)}\right|\dd r\right) \nu(\dd z)
&=\int_{\R}\int_{\R} f_{\sumel \mathbf{m}(1-t)}(z-r)f_{t(\sumel \mathbf{m})}(r)\dd r \frac{\nu(\dd z)}{f_{\sumel \mathbf{m}}(z)} \nonumber \\
&=\int_{\R}\frac{f_{\sumel \mathbf{m}}(z)}{f_{\sumel \mathbf{m}}(z)}\nu(\dd z)\nonumber \\
&=\nu(\R) = 1.
\end{align*}
Since $\R$ is a $\sigma$-finite measure space (with respect to Lebesgue measure), and $f$ is measurable, we can use Fubini's theorem and write
\begin{align*}
\int_{\R}\left(\int_{\R} \left|\frac{f_{\sumel \mathbf{m}(1-t)}(z-r)f_{t(\sumel \mathbf{m})}(r)}{f_{\sumel \mathbf{m}}(z)}\right|\dd r\right) \nu(\dd z)
&=\int_{\R}\left(\int_{\R} \left|\frac{f_{\sumel \mathbf{m}(1-t)}(z-r)f_{t(\sumel \mathbf{m})}(r)}{f_{\sumel \mathbf{m}}(z)}\right|\nu(\dd z)\right) \dd r \nonumber \\
&=\E^{\widetilde{\Q}}\left(|\Theta_t(R_t) |\right).
\end{align*}
Also, since $\{\Theta_t(R_t)\}_{0\leq t<1}$ is harmonic, $\{\Theta_t(R_t)\}_{0\leq t<1}$ is an $(\F_t^{\xib},\widetilde{\Q})$-martingale. 
Explicitly, we have
\begin{align*}
	\E^{\widetilde{\Q}}\left(\Theta_t(R_t) \left|\, \F_s^{\xib} \right.\right)
			&=\E^{\widetilde{\Q}}\left(\left.\int_{-\infty}^{\infty}\frac{f_{\sumel \mathbf{m}(1-t)}(z-R_s-(R_t-R_s))}{f_{\sumel \mathbf{m}}(z)}\, \nu(\dd z)\,\right|\xib_{s} \right) \nonumber
			\\	&=\int_{-\infty}^{\infty}\int_{-\infty}^{\infty}f_{\sumel \mathbf{m}(1-t)}(z-R_s-y)f_{\sumel \mathbf{m}(t-s)}(y)\d y \, \frac{\nu(\dd z)}{f_{\sumel \mathbf{m}}(z)} \nonumber
	\\	&=\int_{-\infty}^{\infty}\frac{f_{\sumel \mathbf{m}(1-s)}(z-R_s)}{f_{\sumel \mathbf{m}}(z)} \, \nu(\dd z) \nonumber
	\\ 	&=\Theta_s(R_s),
\end{align*}
for $0\leq s<t<1$. Since $\Theta_0(R_0)=1$ and $\Theta_t(R_t)>0$, $\{\Theta_t(R_t)\}_{0\leq t<1}$ is a Radon-Nikodym density process. 
We continue by verifying that under $\Q$ the transition law of $\{\xib_t\}_{0\leq t < 1}$ is that of a GLP with generating law $\nu$ and parameter vector $\mb$:
	\begin{align}		
		\Q\left(\xib_t \in \dd \mathbf{x} \,|\, \F_s^{\xib}\right)
				&=\E({1}_{\{\xib_t\in\dd\mathbf{x}\}} \,|\, \F_s^{\xib}) \nonumber
		\\ 	&=\frac{1}{\Theta_s(R_s)}\,\E^{\widetilde{\Q}}(\Theta_t(R_t) {1}_{\{\xib_t\in\dd\mathbf{x}\}} \,|\, \xib_s) \nonumber
		\\  &=\frac{\Theta_{t}(R_t)}{\Theta_{s}(R_{s})}\prod_{i=1}^n f_{(t-s)(m_i)}(x_i-\xi^{(i)}_s) \d x_i.
	\label{eq:MClevy}
	\end{align}
	Comparing equations (\ref{eq:MClevy}) and (\ref{eq:GLP2}) completes the proof.
\end{proof}

\subsection{Proposition \ref{prop:marginaltransition}}
\begin{proof}
Fix $0 \leq s < t < 1$. Then,
\begin{align} \label{eq:marginaltransitionprob}
\Q(\xi^{(i)}_t \in \d y_i | \xi^{(i)}_s=x_i)&= \frac{\int_{\R} \Q(\xi^{(i)}_t \in \d y_i, \xi^{(i)}_s\in \dd x_i | R_1=r) \Q(R_1\in \dd r)}{\int_{\R}\Q(\xi^{(i)}_s\in \dd x_i | R_1=r) \Q(R_1\in \dd r)}.
\end{align}
The numerator of (\ref{eq:marginaltransitionprob}) is
\begin{align}
\label{margnum}
\int_{\R} \Q(\xi^{(i)}_t \in \d y_i, \,&\xi^{(i)}_s\in \dd x_i | R_1=r) \Q(R_1\in \dd r) \notag \\
&= f_{m_is}(x_i)\dd x_i f_{(t-s)m_i}(y_i-x_i) \d y_i \int_{\R} \frac{ f_{\sumel \mathbf{m} -m_it}(r-y_i)}{f_{\sumel \mathbf{m}}(r)}\nu(\dd r),
\end{align}
and the denominator is
\begin{align}
\label{margdenom}
\int_{\R}\Q(\xi^{(i)}_s\in \dd x_i | R_1=r) \Q(R_1\in \dd r) = f_{m_is}(x_i)\dd x_i \int_{\R} \frac{f_{\sumel \mathbf{m} -m_is}(r-x_i)}{f_{\sumel \mathbf{m}}(r)}\nu(\dd r).
\end{align}
Dividing (\ref{margnum}) by (\ref{margdenom}) concludes the first part. 

For the second part, write $\xib_t^{\oslash i}$ for the vector $\xib_t$ excluding its $i$th coordinate. 
Using the Markov property of $\{\xib_t\}_{0\leq t \leq 1}$, we have 
\begin{align}
\Q\left(\xi^{(i)}_t \in \dd y_i \,|\, \F_s^{\xib} \right) = \frac{\Q\left(\xi^{(i)}_t \in \dd y_i,\xi^{(i)}_s \in \dd x_i, \xib_s^{\oslash i} \in \dd \mathbf{x} \right)}{\Q\left(\xi^{(i)}_s \in \dd x_i, \xib_s^{\oslash i} \in \dd \mathbf{x} \right)},
\label{eq:num_denom}
\end{align} 
The numerator of (\ref{eq:num_denom}) is given by
	\begin{multline}
		\label{eq:marginalGLPnum}
		\int_{\R} \Q\left(\left. \xi^{(i)}_t \in \dd y_i , \xi^{(i)}_s \in \dd x_i, \xib_s^{\oslash i} \in\dd \mathbf{x}^{\oslash i}  \,\right| R_1=r \right) \, \Q(R_1\in\dd r)=
		\\ \prod_{j=1}^n[f_{m_js}(x_j) \d x_j] f_{m_i(t-s)}(y_i-x_i) \d y_i  
		\\ \times \int_{\R} \frac{f_{\sumel \mathbf{m}(1-s) + m_i(t-s)}(r-\sum^{n}_{j=1} x_{j} + (y_i-x_i))}{f_{\sumel \mathbf{m}}(r)} \, \nu(\dd r),
	\end{multline}
and the denominator is given by
	\begin{equation}
		\label{eq:denomlevy}
		\Q\left(\xib_s\in\d\mathbf{x} \right) =
		 \prod_{i=1}^{n}[f_{m_is}(x_i)\d x_i] \int_{-\infty}^{\infty} \frac{f_{\sumel \mb(1-s)}(r-\sum^{n}_{i=1} x_{i})}{f_{\sumel \mb}(r)} \,\nu(\dd r).
		\end{equation}
	Equation (\ref{eq:denomlevy}) follows from the stationary increments property of LRBs and (\ref{eq:LRBfinitedimensionaldistributionglp}).	
 Dividing (\ref{eq:marginalGLPnum}) by (\ref{eq:denomlevy}) concludes the second part.
\end{proof}

\subsection{Proposition \ref{prop:R_GLP}}
\begin{proof}
	Since $\{\xib_t\}_{0\leq t \leq 1}$ is a Markov process with respect to $\{\F_t^{\xib}\}_{0\leq t \leq 1}$, $\{R_t\}_{0\leq t \leq 1}$ is a Markov process with respect to $\{\F_t^{\xib}\}_{0\leq t \leq 1}$.
	We first verify (\ref{rcondone}), the $\xib_s$-conditional law of $R_1$.
	For $s=0$, trivially the law of $R_1$ is $\nu$.
	For $0< s <1$, using (\ref{eq:denomlevy}), we have
	\begin{align*} 
		\Q(R_1\in \dd r\,|\, \xib_s=\mathbf{x}) &= \frac{\Q(\xib_s\in \d\mathbf{x} \,|\, R_1=r)\,\nu(\dd z)}{\Q\left(\xib_s\in\d\mathbf{x} \right)} \notag \\
		&= \frac{f_{\sumel \mathbf{m}(1-s)}(r-\sumel \xb) \, \frac{\nu(\dd r)}{f_{\sumel \mathbf{m}}(r)}}
			{\int_{\R}f_{\sumel \mathbf{m}(1-s)}(r-\sumel \xb) \, \frac{\nu(\dd r)}{f_{\sumel \mathbf{m}}(r)}},
	\end{align*}
	as required.
	Next, we verify (\ref{rcondtwo}), the $\xib_s$-conditional law of $R_t$ for $0\leq s < t<1$.
	The process $\{R_t\}_{0\leq t \leq 1}$ is a $\widetilde{\Q}$-\levy process with $\widetilde{\mathbb{P}}(R_{t}\in \d r)=f_{(\sumel\mb) t}(r)\d r$, where $\widetilde{\Q}$ is given by (\ref{eq:RNDPlevy}). 
	Using Corollary \ref{radonnikodymglp} (or \cite{13}(Proposition 3.7)), we have $\{R_t\}_{0\leq t < 1}$ a $\P$-LRB, where
	\begin{align*}		
		\Q\left(R_t \in \dd r \,|\, \xib_s =\mathbf{x}\right)
		&=\Theta_{s}(r)^{-1}\,\E^{\widetilde{\Q}}\left(\Theta_t(r) {1}_{\{R_t\in\dd r\}} \,|\, R_s = \sumel\mathbf{x}\right) \nonumber \\
		&=\Theta_{s}(r)^{-1}\,\int_{\R} \frac{f_{\sumel \mathbf{m}(1-t)}(z-r)}{f_{\sumel \mathbf{m}}(z)} \, \nu(\dd z) f_{(t-s)\sumel\mb}(r-\sumel\mathbf{x})\d r,
	\end{align*}	
	as required.
\end{proof}

\subsection{Proposition \ref{harnessprop}}
\begin{proof}
	Conditional on $\xib_d$ $(0<d\leq 1)$, the coordinates of $\{\xib_t\}_{t\leq d}$ are (independent) \levy bridges, and $\{R_t\}_{t\leq d}$ is a \levy bridge.
	Thus, it is sufficient to prove that an integrable \levy bridge is a harness.
	Let $\{X_t\}_{0\leq t \leq 1}$ be a \levy process such that $X_t$ has a density $f_t$ for $t\in(0,1]$.
	We shall show that the conditional process, and \levy bridge, $\{X_t|X_1=k\}$ is a harness.
	The conditions of the proposition allow us to assume that $\{X_t|X_1=k\}$ is integrable.
	We start by computing the following:
	\begin{equation}
		\Q\left[\bigcap_{i=1}^{n_y} X_{t_i}\in\dd y_i\left|\,
			\left(\bigcap_{i=1}^{n_x} X_{a_i}=x_i \right) \cap \left(\bigcap_{i=1}^{n_z} X_{d_i}=z_i \right) \cap (X_1=k)\right.\right],
				\label{eq:bridge_harness}
	\end{equation}
	for any $n_x, n_y, n_z\in\N_+$, any $0=a_0<a_1<\cdots<a_{n_x}=a<t_1<\cdots<t_{n_y}<d=d_1<\cdots<d_{n_z}<1$, any $(x_1,\ldots,x_{n_x}) \in\R^{n_x}$,
	any $(y_1,\ldots,y_{n_x}) \in\R^{n_y}$, and any $(z_1,\ldots,z_{n_z}) \in\R^{n_z}$.
	Following the Bayes rule, the numerator is
	\begin{align*}
		I_1&=\Q\left[\left(\bigcap_{i=1}^{n_y} X_{t_i}\in\dd y_i\right) \cap
			\left(\bigcap_{i=1}^{n_x} X_{a_i}\in\d x_i \right) \cap \left(\bigcap_{i=1}^{n_z} X_{d_i}\in\d z_i \right) \cap (X_1\in \d k)\right] \nonumber \\
			&= \left(\prod_{i=1}^{n_x} f_{a_i-a_{i-1}}(x_i-x_{i-1}) \dd x_i \right) \\
			&\qquad \times \left(f_{t_1-a}(y_1-x_a)\dd y_1 \prod_{i=2}^{n_y} f_{t_i-t_{i-1}}(y_i-y_{i-1}) \d y_i\right) \nonumber\\
		  &\qquad \times\left(f_{d-t_n}(z_1-y_{n_y})\dd z_1 \prod_{i=2}^{n_z} f_{d_i-d_{i-1}}(z_i-z_{i-1}) \dd z_i\right) f_{1-d_{n_z}}(k-z_{n_z})\d k,
	\end{align*}
	and the denominator is
	\begin{align*}
		I_2&=\Q\left[\left(\bigcap_{i=1}^{n_x} X_{a_i}\in \d x_i \right) \cap \left(\bigcap_{i=1}^{n_z} X_{d_i}\in \d z_i \right) \cap (X_1 \in \d k)\right] \nonumber \\
			&= \left(\prod_{i=1}^{n_x} f_{a_i-a_{i-1}}(x_i-x_{i-1}) \dd x_i \right) \nonumber\\ 
			&\qquad \times \left(f_{d-a}(z_1-x_{n_x})\dd z_1 \prod_{i=2}^{n_z} f_{d_i-d_{i-1}}(z_i-z_{i-1}) \dd z_i\right) f_{1-d_{n_z}}(k-z_{n_z})\d k.
	\end{align*}
	So (\ref{eq:bridge_harness}) is equal to
	\begin{equation*}
		\frac{I_1}{I_2}=\prod_{i=2}^{n_y} \left(f_{t_i-t_{i-1}}(y_i-y_{i-1}) \d y_i\right)
									\frac{f_{t_1-a}(y_1-x_a)f_{d-t_n}(z_1-y_{t_{n_y}})\dd y_1}{f_{d-a}(z_1-x_{n_x})}.
	\end{equation*}
	It follows from the Kolgomorov Extension Theorem that $\{X_t|\mathcal{H}_{a,d}^X\}_{a\leq t \leq d}$ is a \levy bridge between $X_a$ and $X_d$.
	Define $\{\eta_t\}_{0\leq t \leq d-a}$ by $\eta_t=X_{a+t}-X_a$.
	Then $\{\eta_t|\mathcal{H}_{a,d}^X\}$ is \ \levy bridge from $0$ to $X_d-X_a$, and
	\begin{equation*}
		\E[\eta_t|\mathcal{H}_{a,d}^X]= \frac{t}{d-a}(X_d-X_a),
	\end{equation*}
	which yields the result.
\end{proof}

\subsection{Proposition \ref{measureSDEAll}}
\begin{proof}
Define the mapping $H: \R \times \R \times [0,1) \times \R_+ \rightarrow \R_+$ as follows:
\begin{align*}
H(z, y, t, m) = \exp\left\{\frac{z y-t z^2/2}{m^2(1-t)}\right\}.
\end{align*}
Since the Brownian bridges $\{\b^{(i)}_t\}_{0\leq t \leq 1}$, $i=1,\ldots,n$, in (\ref{blpindependentrep}) are mutually independent and $\{ \xib_{t} \}_{0\leq t \leq 1}$ is Markov, we have
\begin{align}
\label{numdem}
\Q(\xib_{1} \in \dd \textbf{x} \,|\, \mathcal{F}_t^{\xib}) &\law \frac{\prod_{i=1}^n H\left(x^{(i)}, \xi_{t}^{(i)}, t, \sigma\sqrt{m_i} \right) \, \Q(\xib_{1} \in \dd \textbf{x})}
									{\int_{\mathbb{R}^n}\prod_{i=1}^n H\left(x^{(i)}, \xi_{t}^{(i)}, t, \sigma\sqrt{m_i} \right) \, \Q(\xib_{1} \in \dd \textbf{x})}  \notag \\
&= \frac{\exp\left\{\sum_i^n \frac{x^{(i)} \xi_{t}^{(i)}-t (x^{(i)})^2/2}{\sigma^2 m_i(1-t)}\right\} \, \Q(\xib_{1} \in \dd \textbf{x})}
									{\int_{\mathbb{R}^n}\exp\left\{\sum_i^n \frac{x^{(i)} \xi_{t}^{(i)}-t (x^{(i)})^2/2}{\sigma^2 m_i(1-t)}\right\} \, \Q(\xib_{1} \in \dd \textbf{x})}.													
\end{align}
If we define the numerator of (\ref{numdem}) as the function
\begin{align}
\label{numalt}
g\left((\xi^{(i)}_{t})_{i=1,\ldots,n},t ; \dd \textbf{x} \right) = \exp\left\{\sum_i^n \frac{x^{(i)} \xi_{t}^{(i)}-t (x^{(i)})^2/2}{\sigma^2 m_i(1-t)}\right\} \, \Q(\xib_{1} \in \dd \textbf{x}),
\end{align}
and apply It\^o's formula to (\ref{numalt}), we get
\begin{align*}
\dd g &= \frac{\partial g}{\partial t} \dd t + \sum_{i=1}^n \frac{\partial g}{\partial \xi_{t}^{(i)}} \dd \xi_{t}^{(i)} +\frac{1}{2}\sum_{i=1}^n \frac{\partial^2 g}{\partial (\xi_{t}^{(i)})^2} \dd \left\langle \xi^{(i)}_{t}\right\rangle + \sum_{i\neq j}^n \frac{\partial^2 g}{\partial \xi_{t}^{(i)}\partial \xi_{t}^{(j)}} \dd \left\langle \xi^{(i)}_{t}, \xi_{t}^{(j)}\right\rangle, \notag \\
&= g\left((\xi^{(i)}_{t})_{i=1,\ldots,n},t ; \dd \textbf{x} \right) \left(\sum_{i=1}^n \frac{x^{(i)}\xi^{(i)}_{t}}{(\sigma^2m_i(1-t)^2}\dd t + \sum_{i=1}^n\frac{x^{(i)}}{\sigma^2 m_i(1-t)}\dd \xi^{(i)}_{t} \right),
\end{align*}
where the covariation brackets $\left\langle \xi^{(i)}_{t}, \xi^{(j)}_{t} \right\rangle$ for $i\neq j$ disappear since the $\{\b^{(i)}_t\}_{0\leq t \leq 1}$, $i=1,\ldots,n$, are mutually independent. Let $G\left((\xi^{(i)}_{t})_{i=1,\ldots,n},t \right) = \int_{\mathbb{R}^n}g\left((\xi^{(i)}_{t})_{i=1,\ldots,n},t ; \dd \textbf{x} \right)$; then, using Fubini's theorem,
\begin{align*}
\dd G\left((\xi^{(i)}_{t})_{i=1,\ldots,n},t \right) = G\left((\xi^{(i)}_{t})_{i=1,\ldots,n},t \right) \left(\sum_{i=1}^n \frac{\E(\xi_1^{(i)}\,|\, \xib_t)\xi^{(i)}_{t}}{\sigma^2 m_i (1-t)^2} + \sum_{i=1}^n\frac{\E(\xi_1^{(i)}\,|\, \xib_t)}{\sigma^2 m_i (1-t)}\dd\xi^{(i)}_{t}\right),
\end{align*}
The statement follows by applying It\^o's formula to the ratio (\ref{numdem}), where we get
\begin{align}
\label{measurevaluedsdeexp}
\dd \phi_t(\dd \textbf{x}) &= \phi_t(\dd \textbf{x})\left(\sum_{i=1}^n \frac{x^{(i)}-\E(\xi_1^{(i)}\,|\, \xib_t)}{\sigma^2 m_i (1-t)} \left(\dd \xi^{(i)}_t - \frac{\E(\xi_1^{(i)}\,|\, \xib_t) - \xi^{(i)}_t}{(1-t)}\right) \right) \notag \\
&\triangleq\phi_t(\dd \textbf{x})\left(\sum_{i=1}^n \sigma_t^{(i)}\dd \tilde{B}_t^{(i)}\right) \notag.
\end{align}
Writing $\tilde{\textbf{B}}_t=(\tilde{B}_t^{(1)},\ldots,\tilde{B}_t^{(n)})^\tp$, define $\{\textbf{B}_t\}_{0\leq t < 1}$ by $\tilde{\textbf{B}}_t=\sigma \sqrt{\mathbf{m}}\circ\textbf{B}_t$. That is, $B_t^{(i)}=\tilde{B}_t^{(i)}/\sigma\sqrt{m_i}$.
For each $i\in\{1,\ldots, n\}$, $\{\xi^{(i)}_t\}_{0\leq t \leq 1}$ is an LRB, and so following similar steps to the proof of  Proposition \ref{canonicalrep}, we can show that $\{B_t^{(i)}\}_{0\leq t < 1}$ is continuous with quadratic variation $t$ and is an $(\mathcal{F}_t^{\xib},\Q)$-martingale.
Then, from L\'evy's characterisation, $\{\textbf{B}_t\}_{0\leq t < 1}$ is a vector of standard $(\mathcal{F}_t^{\xib},\Q)$-Brownian motions.
\end{proof}

\subsection{Proposition \ref{poiossonmarg}}
\begin{proof}
Let $\{\xi_t^{(i)}: i=1,\ldots,n\}_{0\leq t \leq 1}$ be the coordinates of the Poisson Liouville process $\{\xib_t\}_{0\leq t \leq 1}$. 
The survival function of $T^{(i)}$ is
\begin{align*}
	\P(T^{(i)}>s) &= \P(\xi^{(i)}_s = 0) \notag
	\\ &= \E\left( \P\left(\left.\xi_s^{(i)} = 0 \,\right| \sumel \xib_1 \right)\right) \notag
	\\ &= \E\left( (1-s p_i)^{\sumel \xib_1} \right) \notag
	\\ &= G_{\nu}(1-s p_i).
\end{align*}
For $\mathbf{s}\in [0,1]^n$, the joint survival function of $\mathbf{T}$ is
\begin{align*}
	\P(T^{(i)}> s_i; i=1,\ldots,n)&=\P(\xi^{(i)}_{s_i}=0; i=1,\ldots,n) \notag
	\\ &=\E\left(\P(\xi^{(i)}_{s_i}=0; i=1,\ldots,n \,|\, \xib_1)\right) \notag
	\\ &=\E\left( \prod_{i=1}^n \P(\xi^{(i)}_{s_i}=0 \,|\,\xi^{(i)}_1) \right) \notag
	\\ &=\E\left( \prod_{i=1}^n(1-s_i)^{\xi^{(i)}_1} \right) \notag
	\\ &=\E\left( \left( \sum_{i=1}^n p_i(1-s_i) \right)^{\sumel \xib_1} \right) \notag
	\\ &=G_{\nu}\left( \sum_{i=1}^n p_i(1-s_i) \right) \notag
	\\ &=G_{\nu}\left( 1-\sum_{i=1}^n p_i s_i \right),
\end{align*}
which gives the statement.
\end{proof}

%
%
%
%

\end{document}